\documentclass[12pt,a4paper,reqno]{amsart}
 \usepackage{amssymb}
 \usepackage{bbm}
 \usepackage{amsthm}
 \usepackage{amsmath}
 \usepackage{a4wide}
\usepackage{xcolor}
 \usepackage{graphicx}
\usepackage[colorlinks=true, allcolors=blue]{hyperref}
\usepackage{mathabx} 
\usepackage{orcidlink}
\usepackage{cite}

\newtheorem{theorem}{Theorem}[section]
\newtheorem{proposition}[theorem]{Proposition}

\newtheorem{lemma}[theorem]{Lemma}

\theoremstyle{definition}
\newtheorem{example}[theorem]{Example}
\newtheorem{definition}[theorem]{Definition}
\newtheorem*{condS}{Condition~\bf{S}}

\newcommand{\II}{{\mathbb I}} 
\newcommand{\RR}{{\mathbb R}} 
\newcommand{\NN}{{\mathbb N}} 
\newcommand{\ZZ}{{\mathbb Z}} 

\def\cint#1{\left[{#1}\right]} 
\def\rint#1{\left]{#1}\right]} 
\def\opint#1{\left]{#1}\right[} 
\def\uint{\left[0,1\right]} 
\def\ouint{\left]0,1\right[} 

\newcommand{\R}{\mathcal R} 
\newcommand{\D}{\mathcal D} 
\newcommand{\POk}{P_{\mathrm{O},k}} 
\newcommand{\POkD}{P_{\mathrm{O},k,\D}} 
\newcommand{\PMk}{P_{\mathrm{M},k}} 
\newcommand{\PMkD}{P_{\mathrm{M},k,\D}} 
\newcommand{\gkD}{\gamma_{k,\D}} 
\newcommand{\dkD}{\delta_{k,\D}} 

\newcommand{\x}{{\mathbf x}}
\newcommand{\y}{{\mathbf y}}
\newcommand{\dd}{{\mathbf d}}
\newcommand{\vv}{{\mathbf v}}

\newcommand{\uu}{{\mathbf u}}

\newcommand{\ver}{\operatorname{ver}} 
\newcommand{\sign}{\operatorname{sign}} 
\newcommand{\Lk}[1]{L_k^{(#1)}} 
\newcommand{\LkAB}{\Lk{A,B}} 
\newcommand{\LkCB}{\Lk{C,B}} 
\newcommand{\LkAC}{\Lk{A,C}} 
\newcommand{\skb}{{R^*}} 
\newcommand{\dub}{\mathbf R} 
\newcommand{\POkDCB}{P_{\mathrm{O},k,\D}^{(C,B)}} 
\newcommand{\PMkDAC}{P_{\mathrm{M},k,\D}^{(A,C)}} 
 \newcommand{\1}{{\mathbf 1}}

\numberwithin{equation}{section}

\title[Consistency results for standardized functions and semicopulas]{Coherence and avoidance of sure loss for standardized functions and semicopulas}

\author[{\fontsize{7.5}{9}\selectfont E. P. Klement}]{Erich Peter Klement \orcidlink{0000-0002-2054-3286} 
}
\address{Johannes Kepler University Linz, Institute for Mathematical Methods in Medicine and Data Based Modeling, Linz, Austria}
\email{ep.klement@jku.at}

\author[D. Kokol Bukovšek]{Damjana Kokol Bukovšek \orcidlink{0000-0002-0098-6784} }
\address{University of Ljubljana, School of Economics and Business, and Institute of Mathematics, Physics and Mechanics, Ljubljana, Slovenia}
\email{Damjana.Kokol.Bukovsek@ef.uni-lj.si}

\author[B. Mojškerc]{Blaž Mojškerc \orcidlink{0000-0001-8096-355X} }
\address{University of Ljubljana, School of Economics and Business, and Institute of Mathematics, Physics and Mechanics, Ljubljana, Slovenia}
\email{Blaz.Mojskerc@ef.uni-lj.si}

\author[M. Omladič]{Matjaž Omladič \orcidlink{0000-0001-5383-9203} }
\address{Institute of Mathematics, Physics and Mechanics, Ljubljana, Slovenia}
\email{Matjaz@Omladic.net}

\author[S. Saminger-Platz]{Susanne Saminger-Platz \orcidlink{0000-0002-9606-5751} } 
\address{ Johannes Kepler University Linz, Institute for Mathematical Methods in Medicine and Data Based Modeling, Linz, Austria}
\email{Susanne.Saminger-Platz@jku.at}

\author[N. Stopar]{Nik Stopar \orcidlink{0000-0002-0004-4957} } 
\address{University of Ljubljana, Faculty of Civil and Geodetic Engineering, and Institute of Mathematics, Physics and Mechanics, Ljubljana, Slovenia}
\email{Nik.Stopar@fgg.uni-lj.si}


\begin{document}

\begin{abstract}
We discuss avoidance of sure loss and coherence results for semicopulas and standardized functions, i.e., for grounded, 1-increasing functions with value $1$ at $(1,1,\ldots, 1)$. We characterize the existence of $k$-increasing $n$-variate function $C$ fulfilling $A\leq C\leq B$ for standardized $n$-variate functions $A,B$ and discuss the method for constructing this function. Our proofs also include procedures for extending functions on some countably infinite mesh to  functions on the unit box. We provide a characterization when $A$ respectively $B$ coincides with the pointwise infimum respectively supremum of the set of all $k$-increasing $n$-variate functions $C$ fulfilling $A\leq C\leq B$.
\end{abstract}

\maketitle

\section{Introduction and motivation}\label{se:intro}

In recent literature on statistical reasoning, imprecise probabilities have become one of the main tools for modeling uncertainty,
especially in situations when the use of a precise probability model may be questionable or the exact assessment of probability of events impossible.
This is often the case in decision making with vague, incomplete, or even conflicting information \cite{MonMirMon14a}, and in risk management \cite{ArtDelEbeHea99}. The general theory of imprecise probability \cite{Wal91,AugCooDeCTro14} offers a variety of different models for dealing with imprecision such as lower and upper previsions, lower and upper probabilities, probability boxes, distortion probabilities, capacities, and several others, \cite{Schm23,MonMirPelVic15,MonMirDes20,Sca96}.

An imprecise model is typically required to satisfy some reasonable consistency conditions.
Two such conditions are avoidance of sure loss and coherence that were first introduced for lower and upper probabilities \cite{Wal81} and for lower and upper previsions \cite{Wil07}. In the behavioural interpretation, avoidance of sure loss means that a gambler's assessments of events should not lead to acceptance of bets that would produce net loss, regardless of the outcome.
Coherence, on the other hand, suggests that, given a set of acceptable bets, a gambler should also accept any positive linear combination of these bets.
A major difference between the precise and imprecise setting is that lower and upper probabilities are not additive functions. Instead, they are generally at least monotone with respect to set inclusion, i.e., they are \emph{capacities}.
Capacities as a generalization of additive measures were introduced by Choquet in~\cite{Cho54} (see also~\cite{GilSch95} and note that the original definition given there is less general than the one used nowadays). Together with semicopulas they give rise to the framework of \emph{universal integrals}~\cite{KleMesPap10} providing a common frame for many non-additive integrals~\cite{Den94}, 
including the well-known Choquet integral~\cite{Cho54,GilSch94} and  Sugeno integral \cite{Sug74,SugMur87,MurSug91}. 

Avoidance of sure loss and coherence were recently translated to the setting of cumulative distribution functions 
\cite{PelVicMonMir16}, where imprecision is typically modeled with a probability box, i.e., a set of cumulative distribution functions bounded pointwise from above and from below.

The two notions can also be considered for copulas, which motivated the introduction of imprecise copulas \cite{PelVicMonMir13} as boxes of copulas bounded by two quasi-copulas $\underline{C}$ and $\overline{C}$. 
In this setting, avoidance of sure loss is equivalent to the existence of a copula $C$ satisfying $\underline{C} \le C \le \overline{C}$, where $\le$ denotes the pointwise order,
while coherence is equivalent to the two conditions
\begin{align*}
\sup\{C\mid\underline{C} \le C \le \overline{C}\text{ and } C \text{ is a copula}\} & =\overline{C}, \\
\inf\{C \mid \underline{C} \le C \le \overline{C} \text{ and } C \text{ is a copula}\} & =\underline{C}.
\end{align*}
Favourable topological properties of copulas facilitated the application of discretization techniques
that eventually led to a characterization of the two notions \cite{OmlSto20a,OmlSto22}, given solely in terms of the bounding functions $\underline{C}$ and $\overline{C}$.  The newly developed method, now called the ALGEN method \cite{OmlSto22a}, was later extended to distribution functions and applied to give a description of probability boxes in terms of coherent imprecise copulas \cite{OmlSto20b,OmlSto22}.

It is natural to ask whether this new characterization can now be translated back to the theory of lower and upper probabilities and previsions.
This is one of the motivations for the present paper.
We generalize the ALGEN method to the case where the bounds need not be quasi-copulas.
Instead, they are only assumed to be grounded, $1$-increasing, and have value $1$ at $(1,1,\dots,1)$. We will call such functions \emph{standardized} functions by analogy with \cite{PelVicMonMir16}, where standardized functions were introduced in the setting of distributions. This will allow for our results to be used in the theory of multivariate probability boxes and thus provide a stepping stone towards applications in lower and upper previsions.
In particular, our results can be applied to \emph{semicopulas} (which include distribution functions of capacities with uniform margins \cite{Sca96}, and the representing functions of the envelopes of compatible families of continuous distribution functions \cite{Sto23}), in which case we can omit one of the technical assumptions and also obtain an additional characterization for coherence.
The term semicopula was used for the first time by Bassan and Spizzichino~\cite{BasSpi05} in a statistical context. Semicopulas have been known, in a different context, as \emph{conjunctors} (monotone extensions of the Boolean conjunction with neutral element $1$)~\cite{DurKleMesSem07} or as \emph{t-seminorms}~\cite{SuaGil86}. For (structural) properties of the class of semicopulas  see~\cite{DurMesPap08,DurSem05,DurQueSem06}.

Furthermore, we also adapt the method to work for classes of functions other than copulas. In particular, we focus on the classes of $k$-increasing $n$-quasi-copulas. With $k=2$ this includes the class of supermodular $n$-quasi-copulas. The role of $k$-increasing $n$-quasi-copulas (especially for $k=n-1$) has been investigated in \cite{AriMesDeB17} (see also \cite{AriMesDeB20}), while the importance of supermodular, sometimes also-called $L$-superadditive, functions has long been recognized, see e.g. \cite{AriDeB19,MarOlk79,MueSca00,MarMon05,MarMon08,KleKolMesSam17,SamDibKleMes17,SamKolMesKle20} and the references therein.

The structure of the paper is as follows. In Section \ref{se:preliminaries} we give the necessary definitions and basic properties that will be used throughout the paper.  We state our main results on avoidance of sure loss for standardized functions and semicopulas in Section~\ref{se:maintheorem} and give their proofs in Sections~\ref{se:CbyraisingA}--\ref{se:proofofmaintheorem}. Given two specific functions $A \le B$ we construct a $k$-increasing function $C$ between them on a dense countably infinite mesh by modifying the lower bound~$A$ in Section~\ref{se:CbyraisingA} and by modifying the upper bound~$B$ in Section~\ref{se:CbyloweringB_short}. We extend the function~$C$ to the full unit cube in Section~\ref{se:extension} and collect our findings to conclude the proofs of the main results in Section~\ref{se:proofofmaintheorem}.  Section~\ref{se:cogerence} is dedicated to results on coherence. 

\section{Notions and basic properties}\label{se:preliminaries}

\subsection{On $k$-boxes, multiplicities, and related properties}\label{sse:prel:kboxes_properties}

Throughout the paper we shall denote the unit interval by $\II =\uint$ and we will abbreviate the set $\{1, 2, \dots, n\}$ by $[n]$, where $n$ is an arbitrary positive integer which will be fixed for the whole paper. We also denote the points  $(0,0,\dots,0) \in \II^n$ by $\mathbf{0}$ and $(1,1,\dots,1) \in \II^n$ by $\mathbf{1}$. We will use the terms \emph{increasing} to mean non-decreasing and \emph{decreasing} to mean non-increasing.

\begin{definition}
Choose $k\in\NN$ such that $k\in [n]$. Let $\x=(x_1, x_2, \dots, x_n) \in \II^n$ and $\y=(y_1, y_2, \dots, y_n) \in \II^n$ be two points. A Cartesian product of $n$ closed intervals, i.e., a set of the form
\begin{equation*}
    \cint{\x,\y}=\cint{x_1,y_1}\times\cint{x_2,y_2}\times\dots\times\cint{x_n,y_n}
\end{equation*}
will be called a \emph{$k$-box} if 
    $\bigl|\{i\in[n]\mid x_i<y_i\}\bigr|=k$ and
    $|\{i\in[n]\mid x_i=y_i\}|=n-k$.

The \emph{vertices} of a $k$-box $R=\cint{\x,\y}$ will be denoted by
\begin{equation*}
    \ver R=\ver\cint{\x,\y}=\{x_1,y_1\}\times\{x_2,y_2\}\times\dots\times\{x_n,y_n\}.
\end{equation*}
Putting $m=|\{i\in[n]\mid v_i=x_i\}|$, the \emph{sign} of a vertex $\vv$ of a $k$-box $R=\cint{\x,\y}$ is defined by
\begin{equation*}
    \sign_R(\vv)=(-1)^{m-(n-k)}.
\end{equation*}
The \emph{multiplicity} of an arbitrary point $\uu\in\II^n$ with respect to a $k$-box $R$ is given by
\begin{equation*}
m_R(\uu)=
\begin{cases}
     \sign_R(\uu) & \text{if } \uu \in \ver R,\\
     0 & \text{otherwise.}
\end{cases}
\end{equation*}
\end{definition}

Note that, given a $k$-box $R=\cint{\x,\y}$, for each vertex $\vv$ we have $n-k\le m\le n$. In particular, $\sign_R(\y)=1$, since $m=n-k$ in this case, and  $\sign_R(\x)=(-1)^k$, since $m=n$.

We denote by $\R_k(\II^n)$ the set of all \emph{finite disjoint unions} of $k$-boxes with vertices in $\II^n$.
This means that a typical element $\dub\in\R_k(\II^n)$ is of the form $\dub=\bigsqcup_{j=1}^s R_j$,
where $\{R_j\}_{j=1}^{s}$ is an arbitrary finite family (multi-set) of $k$-boxes with vertices in $\II^n$ and $\bigsqcup$ denotes the formal disjoint union.
We extend the definition of the multiplicity of points from a $k$-box to a finite disjoint union of $k$-boxes $\dub=\bigsqcup_{j=1}^s R_j$ by putting for each $\uu \in \II^n$
\begin{equation*}
m_\dub(\uu)=\sum_{j=1}^s m_{R_j}(\uu).
\end{equation*}
Observe that the \emph{disjoint unions} here are \emph{formal disjoint unions}, i.e., a priori the $k$-boxes~$R_j$ need not be disjoint, we just treat them as such when calculating the multiplicities.

\begin{example}\label{ex:2-boxes}
Figure~\ref{fig:2_box_union} depicts two examples of disjoint unions of $2$-boxes in $\II^3$.
On the left we have a disjoint union of four $2$-boxes, namely, $ABLK$, $CFJG$, $MNTS$, and $OQTR$.
So this union is comprised of two overlapping pairs of boxes that are disconnected from each other.
The multiplicity of the point $T$ is $2$ since it is a vertex of two boxes with corresponding multiplicity $1$.
The points $D$, $E$, $I$, $H$, and $P$ have multiplicity $0$ because they are not vertices of any of the boxes.
All other points have multiplicity $-1$ or $1$.
Note that this union can also be interpreted as a union of ten $2$-boxes, namely, $ABED$, $CDHG$, two copies of $DEIH$, $EFJI$, $HILK$, $MNQP$, $OPSR$, and two copies of $PQTS$.
The multiplicities of the points remain the same in this interpretation. This is always the case, because cutting a box into two has no affect on the multiplicities of its points (the multiplicities cancel where the cut is made).

On the right we have a disjoint union of five $2$-boxes, namely, $ABMK$, $CDHF$, $DHNI$, $JKML$, and $MNPO$, two of them are crossing.
Note that at some points the multiplicities are added, while at others they cancel.
In particular, the point $M$ has multiplicity $3$ and the point $K$ has multiplicity $-2$, while the points $D,H$, and $N$ have multiplicity $0$ due to cancellation.
Furthermore, the points $E$ and $G$ have multiplicity $0$ since they are no vertices. All other points have multiplicity $-1$ or $1$.
\end{example}

\begin{figure}
    \centering
    \includegraphics{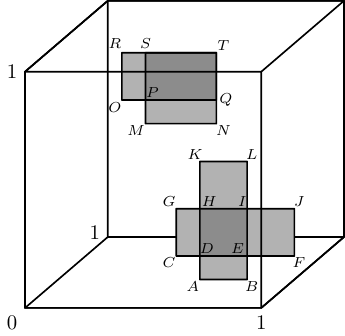} \qquad
    \includegraphics{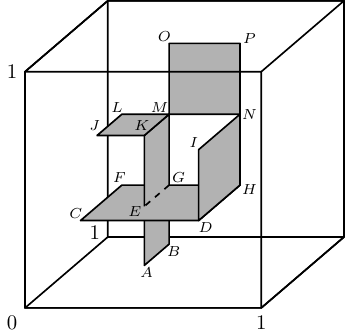}
    \caption{Two examples of a disjoint union of $2$-boxes in $\II^3$ (see Example~\ref{ex:2-boxes}).}
    \label{fig:2_box_union}
\end{figure}

For a $k$-box $R=\cint{\x,\y}$ we have $m_R(\uu)\in\{-1,0, 1\}$ for each $\uu\in\II^n$. Notice that this can also hold for a disjoint union of several $k$-boxes; however, this is a very special case.
In fact, the multiplicity of a point with respect to a finite disjoint union of $k$-boxes can be any integer, as the following lemma shows. 

\begin{lemma}\label{le:KBoxesWithMRFromZ2}
    Let $\x \in \II^n$ be any point which is not a vertex of $\II^n$ and fix some integer $k\in[n]$. 
    Then for every $z\in\ZZ$ there exists a finite disjoint union of $k$-boxes $\dub \in \R_k(\II^n)$ such that $m_{\dub}(\x)=z$.
\end{lemma}

\begin{proof}
For $z=0$  the conclusion is obvious, so suppose that $z\ne 0$. Since $\x$ is not a vertex of the unit cube $\II^n$, there exists at least one coordinate of $\x$ which equals neither $0$ or $1$, and without loss of generality we may assume $x_1 \in \ouint$. For every $j\in[k]\setminus\{1\}$ put
\begin{equation*} a_j = \begin{cases} 1 & \text{if }   x_j = 0, \\ x_j & \text{if }  x_j > 0. \end{cases} 
\end{equation*}
and define the two $k$-boxes
\begin{align*}
    R' &= \cint{(0, 0, \dots, 0, x_{k+1}, \dots,  x_n), (x_1, a_2, \dots, a_k, x_{k+1}, \dots,  x_n)}, \\[0.5ex]
    R'' &= \cint{(x_1, 0, \dots, 0, x_{k+1}, \dots,  x_n), (1, a_2, \dots, a_k, x_{k+1}, \dots, x_n)}. 
\end{align*}
The point $\x$ is a vertex of both $R'$ and $R''$, so $m_{R'}(\x)\in \{-1, 1\}$ and, since they differ only in the first coordinate, $m_{R''}(\x)= -m_{R'}(\x)$.
If $z\cdot m_{R'}(\x) >0$ put $\dub = \bigsqcup_{s=1}^{|z|} R'$, and  $\dub = \bigsqcup_{s=1}^{|z|} R''$ otherwise. In both cases we have $m_{\dub}(\x)=z$.
\end{proof}

\begin{definition}
Let  $k\in\NN$ be such that $k\in[n]$,
and let $A\colon\II^n\to\II$ be an $n$-variate function. Then
\begin{itemize}
    \item[\textup{(i)}] $A$ is \emph{grounded} if $A(\x)=0$ whenever $x_i=0$ for some $i \in [n]$;
    \item[\textup{(ii)}] $A$ has \emph{uniform marginals} if $A(1,\dots,1,x_i,1,\dots,1)=x_i$ for all $x_i \in \II$ and all $i \in [n]$;
    \item[\textup{(iii)}] $A$ is \emph{$k$-increasing} if for every $k$-box $R=\cint{\x,\y}$
\begin{equation*}
   V_{A,k}(R) := \sum_{\vv \in \ver R} \sign_R(\vv)A(\vv) \ge 0.
\end{equation*}
\end{itemize}
\end{definition}
Since the multiplicities of the vertices are additive, for a $k$-increasing function $A\colon\II^n\to\II$ we also have for any disjoint union of $k$-boxes $\dub\in\R_k(\II^n)$ 
\begin{equation*}
 V_{A,k}(\dub) = \sum_{\vv \in \ver \dub} m_\dub(\vv)A(\vv) \ge 0   
\end{equation*}
Note that a function $A$ is $1$-increasing if and only if it is increasing in each variable. The following types of $1$-increasing functions will be of special interest in our paper: standardized functions, semicopulas \cite{DurQueSem06,DurSem05} and quasi-copulas \cite{AlsNelSch93,AriDeB19,AriMesDeB17,AriMesDeB20,GenQueRodSem99} (for some category-related aspects see, e.g., \cite{DurFerTru16}).

\begin{definition}
An $n$-variate function $A\colon\II^n\to\II$ is called
\begin{itemize}
\item[\textup{(i)}] \emph{standardized} if it is grounded, $1$-increasing, and satisfies $A(\mathbf{1})=1$,
\item[\textup{(ii)}] a \emph{semicopula} if it is grounded, $1$-increasing, and has uniform marginals,
\item[\textup{(iii)}] a \emph{quasi-copula} if it is a $1$-Lipschitz semicopula.
\end{itemize}
\end{definition}

We remark that any nonzero, grounded, $1$-increasing function $A\colon\II^n\to\RR$ can be standardized by dividing it by $A(\mathbf{1})$.

We recall the definitions of avoidance of sure loss and coherence for pairs of bivariate standardized functions adjusted to the case of functions defined on $\II^2$.

\begin{definition}[\cite{PelVicMonMir16}]
    Let $A,B \colon \II^2 \to \II$ be standardized functions with $A \le B$.
    \begin{enumerate}
        \item[\textup{(i)}] Given the pair $(A,B)$, we speak about \emph{avoidance of sure loss}  if there exists a $2$-increasing function $C \colon \II^2 \to \II$ such that $A \le C \le B$.
        \item[\textup{(ii)}] The pair $(A,B)$ is \emph{coherent} if
        \begin{align*}
            A &=\inf\{C \colon \II^2 \to \II \mid C \text{ is $2$-increasing, } A \le C \le B\} \text{ and} \\
            B &=\sup\{C \colon \II^2 \to \II \mid C \text{ is $2$-increasing, } A \le C \le B\}.
        \end{align*}
    \end{enumerate}
\end{definition}

We remark that any $2$-increasing function $C$ that satisfies $A \le C \le B$ is in fact a cumulative distribution function, since it is automatically grounded and satisfies $C(1,1)=1$.
Pairs of standardized functions $(A,B)$ satisfying $A \le B$ are called bivariate \emph{probability boxes}, see \cite{PelVicMonMir13}. Bivariate probability boxes can be constructed using so-called imprecise copulas and marginal univariate probability boxes \cite{PelVicMonMir13}.
An \emph{imprecise copula} is a pair of bivariate quasi-copulas $(A,B)$ satisfying $A \le B$. An imprecise copula avoids sure loss if and only if the interval between~$A$ and~$B$ contains a true copula, and it is coherent if and only if~$A$ equals the infimum (and~$B$ the supremum) of all copulas between $A$ and $B$.

We extend these definitions to $n$-variate standardized functions and $k$-increasing $n$-variate functions (coherence of pairs of $n$-quasi-copulas in the case $k=n$ appeared already in \cite{OmlSto22}).

\begin{definition}\label{def:avoidance}
    Let $A,B \colon \II^n \to \II$ be standardized functions with $A \le B$ and $k\in[n]\setminus\{1\}$. 
    \begin{enumerate}
        \item[\textup{(i)}] Given the pair $(A,B)$, we speak about \emph{$k$-avoidance of sure loss} if there exists a $k$-increasing function $C \colon \II^n \to \II$ such that $A \le C \le B$.
        \item[\textup{(ii)}] The pair $(A,B)$ is \emph{$k$-coherent} if
        \begin{align*}
            A &=\inf\{C \colon \II^n \to \II \mid C \text{ is $k$-increasing, } A \le C \le B\} \text{ and} \\
            B &=\sup\{C \colon \II^n \to \II \mid C \text{ is $k$-increasing, } A \le C \le B\}.
        \end{align*}
    \end{enumerate}
\end{definition}

The following lemma, though summarizing a very basic mathematical fact, will facilitate arguments and readability in the later proofs. 
\begin{lemma}\label{lem:claim-green}
Consider two $n$-variate functions $A,B\colon\II^n\to\II$  with $A\le B$ and fix some integer $k\in[n]$. 
For an arbitrary finite disjoint union of $k$-boxes $\dub\in\R_k(\II^n)$ and an arbitrary $\y\in\II^n$ the following holds: 
\begin{equation*}
\max\{m_{\dub}(\y)A(\y),m_{\dub}(\y)B(\y)\} = m_{\dub}(\y)\cdot\begin{cases}
            A(\y) & \text{if }   m_{\dub}(\y) < 0, \\
            0 & \text{if }  m_{\dub}(\y) = 0, \\
            B(\y) & \text{if }  m_{\dub}(\y) > 0.
            \end{cases}
\end{equation*}
\end{lemma}

\begin{definition}\label{prel:def:Lfunction}
Let $A,B \colon \II^n \to \II$ be two functions with $A\le B$ and define the function $\Lk{A,B}\colon\R_k(\II^n)\to\RR$ by
 \begin{equation*}
\Lk{A,B}(\dub) =\sum_{\substack{\y\in\II^n\\m_\dub(\y)>0}} m_\dub(\y)B(\y) + \sum_{\substack{\y\in\II^n\\ m_\dub(\y)<0}} m_\dub(\y)A(\y). 
\end{equation*}
\end{definition}

Because of Lemma~\ref{lem:claim-green} we have for all $\dub \in \R_k(\II^n)$
\begin{equation}\label{eq:LKonD}
\Lk{A,B}(\dub)=\sum_{\y\in\II^n} \max\{m_\dub(\y)A(\y), m_\dub(\y)B(\y)\}. 
\end{equation}
This is actually a finite sum  since $m_\dub(\y) \ne 0$ only for finitely many $\y \in \II^n$. Note that for a disjoint union of $k$-boxes $\dub$ we have $\Lk{A,A}(\dub)=V_{A,k}(\dub)$, and thus the function $A$ is $k$-increasing if and only if  $\Lk{A,A}(\dub) \ge 0$ for all $\dub \in \R_k(\II^n)$.

Furthermore, define the functions $\PMk^{(A,B)}\colon\II^n\to\RR$ and $\POk^{(A,B)}\colon\II^n\to\RR$ by
\begin{equation*}
\PMk^{(A,B)}(\x)=\inf_{\substack{\dub\in\R_k(\II^n)\\ m_\dub(\x)>0}}\frac{\Lk{A,B}(\dub)}{|m_\dub(\x)|} \qquad\text{and}\qquad \POk^{(A,B)}(\x)= \inf_{\substack{\dub\in \R_k(\II^n)\\ m_\dub(\x)<0}} \frac{\Lk{A,B}(\dub)}{|m_\dub(\x)|}.
\end{equation*}
Here we adopt the convention that the infimum of an empty set equals~$0$.
We also define the functions $\gamma_k^{(A,B)}\colon\II^n\to\RR$ and $\delta_k^{(A,B)}\colon\II^n\to\RR$ by
\begin{align*}
    \gamma_k^{(A,B)}(\x)&=\min\{\POk^{(A,B)}(\x),B(\x)-A(\x)\},\\[1ex]
    \delta_k^{(A,B)}(\x)&=\min\{\PMk^{(A,B)}(\x),B(\x)-A(\x)\}.
    \end{align*}
The intuition behind functions defined above is as follows.
The $k$-avoidance of sure loss  for a pair $(A,B)$ of semicopulas will be equivalent to the function $\Lk{A,B}$ being non-negative on any disjoint union of $k$-boxes, see Theorem~\ref{th:maintheorem}.
Now suppose that function $\Lk{A,B}$ is non-negative. The value $\gamma_k^{(A,B)}(\x)$ tells us at most how much we can increase the value of function $A$ at a single point $\x$ so that after the change the function $\Lk{A,B}$ remains non-negative, see Proposition~\ref{pr:step1A}.
Similarly, the value $\delta_k^{(A,B)}(\x)$ tells us at most how much we can decrease the value of function $B$ at point $\x$ so that after the change the function $\Lk{A,B}$ remains non-negative, see Proposition~\ref{pr:step1B}.

\begin{definition}
Let $\delta_1, \delta_2, \dots, \delta_n$ be subsets of $\II$ containing both $0$ and $1$, and put $\D = \prod_{i=1}^n \delta_i \subseteq \II^n$. 
If each set $\delta_i$ is countably infinite and dense in $\II$  then also $\D$ is countably infinite and dense in $\II^n$. We call such a $\D$ a \emph{dense countably infinite mesh} in $\II^n$. 
\end{definition}

Note that a dense countably infinite mesh $\D$ contains the points $\mathbf{0}$ and $\mathbf{1}$.  We define a version of the functions $\PMk^{(A,B)}$, $\POk^{(A,B)}$, $\gamma_k^{(A,B)}$, and $\delta_k^{(A,B)}$ for functions $A, B$ defined on a mesh $\D$ as follows.
Let $\R_k(\D)$ be the set of all finite disjoint unions of $k$-boxes with vertices in $\D$ and let $A,B \colon \D \to \II$ be functions with $A\le B$. Then the functions $\PMkD^{(A,B)}$, $\POkD^{(A,B)}$, $\gkD$ and $\dkD$ all map from $\D$ into $\RR$ and are defined by, respectively,  
\begin{gather*}
\PMkD^{(A,B)}(\dd) =\inf_{\substack{\dub\in\R_k(\D)\\ m_\dub(\dd)>0}}\frac{\Lk{A,B}(\dub)}{|m_\dub(\dd)|}\qquad \text{and} \qquad 
\POkD^{(A,B)}(\dd) = \inf_{\substack{\dub\in \R_k(\D)\\ m_\dub(\dd)<0}} \frac{\Lk{A,B}(\dub)}{|m_\dub(\dd)|},
\\[1ex]
    \gkD^{(A,B)}(\dd) =\min\{\POkD^{(A,B)}(\dd),B(\dd)-A(\dd)\},\\[1ex]
    \dkD^{(A,B)}(\dd) =\min\{\PMkD^{(A,B)}(\dd),B(\dd)-A(\dd)\}.
\end{gather*}
If the point  $\x$ from Lemma~\ref{le:KBoxesWithMRFromZ2} belongs to $\D$, then the disjoint union of $k$-boxes $\dub$ can be chosen from $\R_k(\D)$.
Furthermore, Lemma~\ref{lem:claim-green} is valid also if the functions $A, B$ are defined on $\D$ only and $\y\in\D$, in which case we have for all $\dub \in \R_k(\D)$
\begin{equation*} \Lk{A,B}(\dub)=\sum_{\dd\in\D} \max\{m_\dub(\dd)A(\dd), m_\dub(\dd)B(\dd)\}. 
\end{equation*}

\subsection{Bounds for $\POk^{(A,B)}$ and $\PMk^{(A,B)}$}\label{sse:prel:boundsPOPM} 

\begin{proposition}\label{pr:step2}
Let $\D$ be a dense countably infinite mesh in $\II^n$ and fix some integer $k\in[n]$. 
Let $A, B \colon \D \to \II$ be functions with $A \le B$ and $L_k^{(A,B)}(\dub)\ge 0$ for all $\dub\in \R_k(\D)$. 
Furthermore, assume that $A(\vv) = B(\vv)$ for all vertices $\vv$ of the unit cube $\II^n$. Then for each $\x \in \D$
\begin{equation*}
\POkD^{(A,B)}(\x) + \PMkD^{(A,B)}(\x) \ge B(\x) - A(\x).
\end{equation*} 
\end{proposition}

\begin{proof}
If $\x$ is a vertex of the unit cube $\II^n$ then the claim holds, since the right-hand side of the inequality equals~$0$. 
So fix $\x \in \D$ which is not a vertex of the unit cube $\II^n$. By Lemma~\ref{le:KBoxesWithMRFromZ2} there exist some $\dub_1, \dub_2 \in \R_k(\D)$ with $m_{\dub_1}(\x) < 0$ and $m_{\dub_2}(\x) > 0$ which can be used to define a new disjoint union of $k$-boxes $\dub_3$ by
\begin{equation*} \dub_3 = \left(\bigsqcup_{t=1}^{|m_{\dub_2}(\x)|} \dub_1\right) \bigsqcup \left(\bigsqcup_{t=1}^{|m_{\dub_1}(\x)|} \dub_2\right),
\end{equation*}
i.e., $\dub_3$ consists of $|m_{\dub_2}(\x)|$ copies of $\dub_1$ and $|m_{\dub_1}(\x)|$ copies of $\dub_2$.
We want to show that 
\begin{equation} \label{eq:LkR3a}
\frac{L_k^{(A,B)}(\dub_1)}{|m_{\dub_1}(\x)|} + \frac{L_k^{(A,B)}(\dub_2)}{|m_{\dub_2}(\x)|}  \ge B(\x) - A(\x) + \frac{L_k^{(A,B)}(\dub_3)}{|m_{\dub_1}(\x)|\cdot|m_{\dub_2}(\x)|}.
\end{equation}
Inequality~\eqref{eq:LkR3a} is equivalent to 
\begin{equation}\label{eq:LkR3}
\begin{split} 
|m_{\dub_2}(\x)|\cdot L_k^{(A,B)}(\dub_1) &+ |m_{\dub_1}(\x)|\cdot L_k^{(A,B)}(\dub_2)  \\ 
&\ge |m_{\dub_1}(\x)|\cdot|m_{\dub_2}(\x)|\cdot (B(\x) - A(\x)) + L_k^{(A,B)}(\dub_3),
\end{split}
\end{equation}
which in turn can be rewritten into
\begin{align*}
|m_{\dub_2}(\x)| &\cdot\sum_{\dd\in\D} \max\left\{m_{\dub_1}(\dd)A(\dd), m_{\dub_1}(\dd)B(\dd)\right\} \\
 & \phantom{xxx} + |m_{\dub_1}(\x)|\cdot \sum_{\dd\in\D} \max\left\{m_{\dub_2}(\dd)A(\dd), m_{\dub_2}(\dd)B(\dd)\right\} \\
 &\ge |m_{\dub_1}(\x)|\cdot|m_{\dub_2}(\x)|\cdot (B(\x) - A(\x)) + \sum_{\dd\in\D} \max\left\{m_{\dub_3}(\dd)A(\dd), m_{\dub_3}(\dd)B(\dd)\right\}.
\end{align*}
We shall investigate the contribution of each $\dd\in\D$ to both sides of the inequality~\eqref{eq:LkR3} by distinguishing the two cases (1) $\dd=\x$, and (2) $\dd \ne \x$. Note that the term $|m_{\dub_1}(\x)|\cdot|m_{\dub_2}(\x)|\cdot (B(\x) - A(\x))$ contains the point $\x$, so it needs to be considered only in case (1).

\emph{Case} $1$: $\dd=\x$. Since $m_{\dub_1}(\x) < 0$ and $m_{\dub_2}(\x) > 0$, its contribution to the left-hand side of \eqref{eq:LkR3} is
\begin{equation*}
|m_{\dub_2}(\x)|\cdot m_{\dub_1}(\x)\cdot A(\x) + |m_{\dub_1}(\x)|\cdot m_{\dub_2}(\x)\cdot B(\x) 
= |m_{\dub_1}(\x)|\cdot|m_{\dub_2}(\x)|\cdot (B(\x) - A(\x)).
\end{equation*}
Since $m_{\dub_3}(\x) = |m_{\dub_2}(\x)|\cdot m_{\dub_1}(\x) + |m_{\dub_1}(\x)|\cdot m_{\dub_2}(\x) = 0$, its contribution to the right-hand side of \eqref{eq:LkR3} is also $|m_{\dub_1}(\x)|\cdot|m_{\dub_2}(\x)|\cdot (B(\x) - A(\x))$, and the inequality holds.

\emph{Case} $2$:  $\dd \ne \x$. Since $|m_{\dub_2}(\x)|\cdot m_{\dub_1}(\dd) + |m_{\dub_1}(\x)|\cdot m_{\dub_2}(\dd) = m_{\dub_3}(\dd)$,
for the contribution of $\dd$ to the left-hand side of \eqref{eq:LkR3} we obtain the following lower bound
\begin{align*} |m_{\dub_2}(\x)|&\cdot\max\left\{m_{\dub_1}(\dd) A(\dd), m_{\dub_1}(\dd)B(\dd)\right\} + |m_{\dub_1}(\x)| \cdot \max\left\{m_{\dub_2}(\dd) A(\dd), m_{\dub_2}(\dd)B(\dd)\right\} \\[1ex]
&\ge \max\left\{|m_{\dub_2}(\x)| \cdot m_{\dub_1}(\dd)\cdot A(\dd) + |m_{\dub_1}(\x)| \cdot m_{\dub_2}(\dd)\cdot A(\dd),\right. \\[1ex]
& \phantom{xxxxxx} \left.|m_{\dub_2}(\x)|\cdot m_{\dub_1}(\dd)\cdot B(\dd) + |m_{\dub_1}(\x)|\cdot m_{\dub_2}(\dd)\cdot B(\dd)\right\} \\[1ex]
&= \max\left\{m_{\dub_3}(\dd)A(\dd), m_{\dub_3}(\dd)B(\dd)\right\},
\end{align*}
equaling the contribution of $\dd$ to the right-hand side of \eqref{eq:LkR3}.
Thus, inequality~\eqref{eq:LkR3a} is verified.
The last term of inequality \eqref{eq:LkR3a} is non-negative by assumption, so it follows that 
\begin{equation*} 
\frac{L_k^{(A,B)}(\dub_1)}{|m_{\dub_1}(\x)|} + \frac{L_k^{(A,B)}(\dub_2)}{|m_{\dub_2}(\x)|}  \ge B(\x) - A(\x) .
\end{equation*}
In order to obtain the desired result it suffices to take the infimum over all 
$\dub_1\in \R_k(\D)$ with $m_{\dub_1}(\x) < 0$, on the one hand, and the infimum over all 
$\dub_2\in \R_k(\D)$ with $m_{\dub_2}(\x) > 0$, on the other hand.
\end{proof}

\begin{lemma}\label{lem:cont}
Let $A,B\colon\II^n\to\II$ be semicopulas with $A\le B$ and fix some integer $k\in[n]$. 
If $B$ is continuous then for all $\x\in\II^n$ 
\begin{equation*}
  \POk^{(A,B)}(\x)\le B(\x)-A(\x).  
\end{equation*}
If $A$ is continuous then  for all $\x\in\II^n$
\begin{equation*}
   \PMk^{(A,B)}(\x)\le B(\x)-A(\x). 
\end{equation*}
\end{lemma}

\begin{proof}
First assume that $\x = (x_1, \dots, x_n) \in\rint{0,1}^n$ and that $B$ is continuous. If $\x = \mathbf{1}$, then the set $\{\dub\in\R_k(\II^n) \mid m_\dub(\x)<0\}$ is empty, $\POk^{(A,B)}(\x) =0$ by definition, and the first inequality holds. If $\x \ne \mathbf{1}$, there exists $x_i < 1$. By permuting the coordinates, we may assume without loss of generality that $x_1 < 1$. Choose an $\varepsilon$ such that $0 < \varepsilon \le 1 - x_1$ and denote
$\x' = (x_1, 0, \dots, 0, x_{k+1}, \dots, x_n)$ and $\x'' = (x_1+\varepsilon, x_2, \dots, x_k, x_{k+1}, \dots, x_n)$.
Then $R_1 = \cint{\x', \x''}$ is a $k$-box with $m_{R'}(\x) = -1$ and 
\begin{equation*}
\POk^{(A,B)}(\x) \le \Lk{A,B}(R_1) = B(\x'') - A(\x),
\end{equation*}
since all other vertices of $R_1$ have at least one coordinate which equals~$0$. Sending $\varepsilon$ to 0 and using the continuity of $B$ gives the first inequality for the point $\x$.

Next assume that $\x\in\rint{0,1}^n$ and that $A$ is continuous. Choose an $\varepsilon$ such that $0 < \varepsilon \le x_1$ and put  
$\x' = (x_1 - \varepsilon, 0, \dots, 0, x_{k+1}, \dots, x_n)$  and $\x'' = (x_1-\varepsilon, x_2, \dots, x_k, x_{k+1}, \dots, x_n)$.
Then $R_2 = \cint{\x', \x}$ is a $k$-box with $m_{R_2}(\x) = 1$, $m_{R_2}(\x'') = -1$, and 
\begin{equation*}\PMk^{(A,B)}(\x) \le \Lk{A,B}(R_2) = B(\x) - A(\x'').
\end{equation*} 
Again, by sending $\varepsilon$ to 0 and using the continuity of $A$ we obtain the second inequality for the point $\x$.

Now assume that $\x\in\II^n\setminus\rint{0,1}^n$, so at least one coordinate $x_i$ of $\x$ equals~$0$, implying that $B(\x) - A(\x) =0$. We need to prove that $\POk^{(A,B)}(\x) = \PMk^{(A,B)}(\x) = 0$. We will do this without using any continuity of $A$ or $B$, so the same reasoning will work for both cases. Choose an $\varepsilon>0$. If the set $\{\dub\in\R_k(\II^n) \mid m_\dub(\x)<0\}$ is empty then $\POk^{(A,B)}(\x) =0$ by definition. If it is non-empty then there exists a $k$-box $R_3 = \cint{\x', \x''}$ with $m_{R_3}(\x) = -1$ such that 
 $\x' = (x'_1, \dots, x'_n)$, $\x'' = (x''_1, \dots, x''_n)$  and $x''_j - x'_j \le \varepsilon$  for all $j \in [n]$.
We have 
\begin{equation*}
   \POk^{(A,B)}(\x) \le \Lk{A,B}(R_3) = \sum_{\substack{\y\in\II^n\\m_{R_3}(\y)=1}} B(\y) - \sum_{\substack{\y\in\II^n\\ m_{R_3}(\y)=-1}} A(\y)
   \le \sum_{\substack{\y\in\II^n\\m_{R_3}(\y)=1}} B(\y) \le \sum_{\y\in \ver R_3} B(\y) .
\end{equation*}
Since $x_i=0$ we have $x''_i \le \varepsilon$, so $y_i \le \varepsilon$ for every $\y = (y_1, \dots, y_n) \in \ver R_3$. This means that $B(\y) \le B(1, \dots, 1, y_i, 1, \dots, 1) = y_i \le \varepsilon$  since $B$ is a semicopula. It follows that $\POk^{(A,B)}(\x) \le 2^k \varepsilon$, and sending $\varepsilon$ to $0$ gives $\POk^{(A,B)}(\x) =0$. The equality $\PMk^{(A,B)}(\x) = 0$ is shown similarly.
\end{proof}

\section{Main theorems}\label{se:maintheorem}

In this section we formulate our main results, the proofs of which will be given in Section~\ref{se:proofofmaintheorem}. These results give characterizations of pairs $(A,B)$ of standardized functions and semicopulas such that we have $k$-avoidance of sure loss (see Definition~\ref{def:avoidance}).
Our results extend the ALGEN method to the setting of standardized functions and semicopulas for any $k\in [n]$.
The acronym ALGEN stands for \underline{Al}gebraic Obstacles in the \underline{Ge}ometry of \underline{N}egative Volumes.
It is a method for constructing a copula lying between two given quasi-copulas~$A$ and $B$ with $A\leq B$, if it exists. For more details on the method see \cite[Appendix~A]{OmlSto22a}. In order to state our main results we first introduce the following notion.

\begin{condS} A function $A \colon \II^n \to \II$ satisfies Condition {\bf S} if there exists a countable set $S \subseteq \II$ such that for every $\uu \in \II^n$ and every $i \in [n]$ the set of discontinuities of the section $t \longmapsto A(u_1,\dots,u_{i-1},t,u_{i+1},\dots,u_n)$ is contained in $S$.
\end{condS}

Note that each section $t \longmapsto A(u_1,\dots,u_{i-1},t,u_{i+1},\dots,u_n)$ of a $1$-increasing function $A \colon \II^n \to \II$ has countably many discontinuities.
Condition~{\bf S} requires that for each section its set of discontinuities is
contained in a common countable set $S$.
Examples of functions that do respectively do not satisfy Condition~{\bf S} are depicted in Figure~\ref{fig:cond_S}. Here is our first main result.

\begin{figure}
    \centering
    \includegraphics[width=0.35\linewidth]{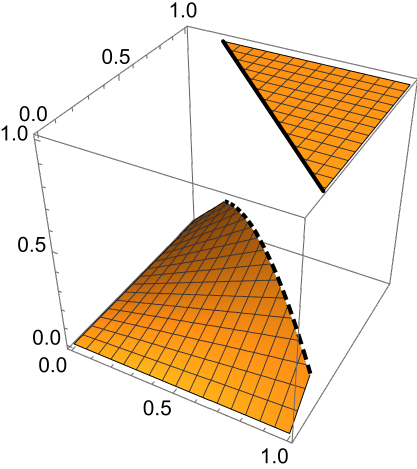} \qquad\includegraphics[width=0.35\linewidth]{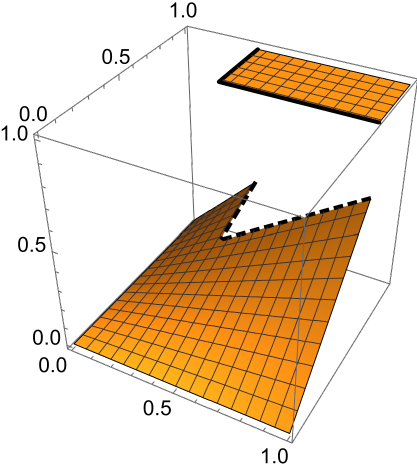}
    \caption{Graph of a function which does not satisfy Condition~{\bf S} (left) and a function which satisfies Condition~{\bf S} with $S=\{\frac 13,\frac 23\}$ (right).}
    \label{fig:cond_S}
\end{figure}

\begin{theorem}\label{th:maintheorem-2}
Let $A,B \colon \II^n \to \II$ be standardized functions with $A \leq B$ and fix some integer $k\in[n]$.
Suppose that at least one of the functions $A$ and $B$ satisfies Condition~{\bf S} for a set $S\subseteq\II$. Then the following statements are equivalent:
\begin{itemize}
    \item[\textup{(i)}] There exists a $k$-increasing $n$-variate function $C\colon\II^n\to\II$ such that $A\le C\le B$.
    \item[\textup{(ii)}] For all $\dub\in \R_k(\II^n): \quad  L_k^{(A,B)}(\dub)\ge 0.$
\end{itemize}
\end{theorem}

Note that whenever $\D\subseteq\II^n$ is a countably infinite mesh then the assertion \textup{(ii)} of Theorem~\ref{th:maintheorem-2} implies that 
$\LkAB(\dub)\ge 0$  for all $\dub\in\R_k(\D)$.
In the framework of semicopulas, Theorem~\ref{th:maintheorem-2} can be strengthened by omitting Condition~\textbf{S}.

\begin{theorem}\label{th:maintheorem}
Consider two $n$-variate semicopulas $A,B\colon \II^n\to\II$ with $A\le B$ and fix some integer $k\in [n]$. 
Then the following statements are equivalent:
\begin{itemize}
    \item[\textup{(i)}] There exists a $k$-increasing $n$-variate semicopula $C\colon\II^n\to\II$ such that $A\le C\le B$.
    \item[\textup{(ii)}] For all $\dub\in \R_k(\II^n): \quad L_k^{(A,B)}(\dub)\ge 0.$
\end{itemize}
\end{theorem}

\section{Construction of $C$ from below and discussion of its properties}\label{se:CbyraisingA}

\subsection{Constructing $C$ from below}\label{sse:Cfrombelow}

We will construct the function $C$ first on a dense countably infinite mesh $\D$. This will be done by raising the values of the function $A$ point by point. The first proposition describes how this is done at a single point.

\begin{proposition} \label{pr:step1A}
Let $\D$ be a dense countably infinite mesh in $\II^n$ and fix some integer $k\in[n]$.
Let $A, B \colon \D \to \II$ be functions with $A \le B$  and $L_k^{(A,B)}(\dub)\ge 0$ for all $\dub\in \R_k(\D)$.  Fix a point $\x \in \D$ and define the function $A' \colon \D \to \II$ by
\begin{equation*}
A'(\uu) =  \begin{cases}
            A(\uu) & \text{if }   \uu \ne \x, \\
            A(\x) + \gkD^{(A,B)}(\x)& \text{if }   \uu = \x.
            \end{cases}
\end{equation*}
Then it follows that $A\le A' \le B$, that the pair $(A', B)$ satisfies the condition $L_k^{(A',B)}(\dub)\ge 0$ for all $\dub\in \R_k(\D)$, and that $\gkD^{(A',B)}(\x) = 0.$
\end{proposition} 

\begin{proof}
If $\uu \ne \x$ then $A'(\uu) = A(\uu) \le B(\uu)$ and, by definition of $\gkD^{(A,B)}$,
\begin{equation*}
A'(\x) = A(\x) + \gkD^{(A,B)}(\x) \le A(\x) + B(\x) - A(\x) = B(\x).
\end{equation*}
If $\dub \in \R_k(\D)$ with $m_\dub(\x) \ge 0$ then $L_k^{(A',B)}(\dub) = L_k^{(A,B)}(\dub) \ge 0$. If $\dub \in \R_k(\D)$ with $m_\dub(\x) < 0$ then 
\begin{equation} \label{eq:LkA'}
L_k^{(A',B)}(\dub) = L_k^{(A,B)}(\dub) - m_\dub(\x)A(\x) + m_\dub(\x)A'(\x) = L_k^{(A,B)}(\dub) + m_\dub(\x)\gkD^{(A,B)}(\x),
\end{equation}
implying 
\begin{equation*} \gkD^{(A,B)}(\x) \le \POkD^{(A,B)}(\x) \le \frac{L_k^{(A,B)}(\dub)}{|m_\dub(\x)|} = - \frac{L_k^{(A,B)}(\dub)}{m_\dub(\x)} .
\end{equation*}
Since $m_\dub(\x) < 0$ it follows that $m_\dub(\x)\gkD^{(A,B)}(\x) \ge - L_k^{(A,B)}(\dub)$, and thus again $L_k^{(A',B)}(\dub) \ge 0$.
To verify $\gkD^{(A',B)}(\x) = 0$ we compute
\begin{equation*} \POkD^{(A',B)}(\x) = \inf_{\substack{\dub\in\R_k(\D)\\ m_\dub(\x)<0}}\frac{\Lk{A',B}(\dub)}{|m_\dub(\x)|} = \inf_{\substack{\dub\in\R_k(\D)\\ m_\dub(\x)<0}} \frac{L_k^{(A,B)}(\dub)}{|m_\dub(\x)|} - \gkD^{(A,B)}(\x)= \POkD^{(A,B)}(\x) - \gkD^{(A,B)}(\x)\end{equation*}
by \eqref{eq:LkA'}. Note that we may use~\eqref{eq:LkA'} since the infimum is taken over disjoint unions of $k$-boxes $\dub$ with $m_\dub(\x) < 0$ only. Finally,
\begin{align*}
\gkD^{(A',B)}(\x) &= \min\{\POkD^{(A',B)}(\x), B(\x) - A'(\x)\} \\[1ex]
&= \min\{\POkD^{(A,B)}(\x) - \gkD^{(A,B)}(\x), B(\x) - A(\x) - \gkD^{(A,B)}(\x)\} \\[1ex]
&= \gkD^{(A,B)}(\x) - \gkD^{(A,B)}(\x) = 0. 
\tag*{\qedhere}
\end{align*}
\end{proof}

In the following proposition we construct $C$ as a pointwise limit of an increasing sequence of functions~$A^{(i)}$ obtained from $A$.

\begin{proposition}\label{pr:step3A}
Let $\D$ be a dense countably infinite mesh in $\II^n$ and fix some integer $k\in[n]$.
Let $A, B \colon \D \to \II$ be functions with $A \le B$ and  $L_k^{(A,B)}(\dub)\ge 0$ for all $\dub\in \R_k(\D)$. Then there exists a function $C\colon\D\to\II$ such that
\begin{itemize}
	\item[\textup{(i)}] $A \le C \le B$ on $\D$,
	\item[\textup{(ii)}] $\gkD^{(C,B)}(\dd) = 0$ for all $\dd \in \D$,
	\item[\textup{(iii)}] $\LkCB(\dub) \ge 0$ for all $\dub\in\R_k(\D)$.
\end{itemize}
\end{proposition}

\begin{proof}
Since $\D$ is countable we can arrange the elements of $\D$ into a sequence $(\dd_i)_{i\in \NN}$. We recursively define a sequence of functions $A^{(i)} \colon \D \to \II$ putting $A^{(0)}= A$ and, for $i\ge 1$,
\begin{equation}\label{eq:step3A:DefAis}
A^{(i)}(\dd) =  \begin{cases}
            A^{(i-1)}(\dd) & \text{if }  \dd \ne \dd_i, \\
            A^{(i-1)}(\dd_i) + \gkD^{(A^{(i-1)},B)}(\dd_i) & \text{if }   \dd = \dd_i.
            \end{cases}
\end{equation}
The definition of $A^{(i)}$ and Proposition~\ref{pr:step1A}, imply $A \le A^{(1)} \le A^{(2)} \le \cdots \le A^{(i-1)} \le A^{(i)} \le~\cdots$ and $A^{(i)} \le B$. Using Proposition~\ref{pr:step1A}, we also have $\gkD^{(A^{(i)},B)}(\dd_i) = 0$ for all $i \in \NN$ and $\Lk{A^{(i)},B}(\dub) \ge 0$ for all $\dub\in\R_k(\D)$. It follows that $\gkD^{(A,B)}(\dd) \ge \gkD^{(A^{(1)},B)}(\dd) \ge \dots \ge \gkD^{(A^{(i)},B)}(\dd) \ge \cdots$ for all $\dd \in \D$. Hence, $\gkD^{(A^{(i)},B)}(\dd_j)=0$ for all $i \ge j$.

Now, let $C$ be the pointwise limit of the sequence $(A^{(i)})_{i\in\NN}$. The limit exists since at each $\dd \in \D$ the sequence of numbers $(A^{(i)}(\dd))_{i\in\NN}$ is increasing and bounded. It immediately follows that.
\begin{itemize}
	\item[\textup{(i)}] $A \le C \le B$ on $\D$;
	\item[\textup{(ii)}] $\gkD^{(C,B)}(\dd_j) = 0$ for all $j\in \NN$;
	\item[\textup{(iii)}] $\LkCB(\dub)= \lim_{i \to \infty} \Lk{A^{(i)},B}(\dub)\ge 0$ for all $\dub\in\R_k(\D)$; 
\end{itemize}
where~\textup{(iii)} holds because there are only finitely many points $\dd \in \D$ with $m_\dub(\dd) \ne 0$, thus completing the proof.
\end{proof}

\subsection{On the $k$-increasingess of $C$}\label{sse:step4A}

We have so far shown that for two $n$-variate functions $A,B\colon\D\to\II$ on a dense countably infinite mesh $\D\subseteq\II^n$ with $\LkAB(\dub)\ge 0$ for all $\dub\in\R_k(\D)$ we may obtain another function $C\colon\D\to\II^n$ as the pointwise limit of a sequence of functions $(A^{(i)})_{i\in \NN}$ as given by \eqref{eq:step3A:DefAis}. 

We shall show that the  function $C$ obtained in this way is $k$-increasing, i.e., fulfills $\Lk{C,C}(\dub)\ge 0$ for all $\dub\in\R_k(\D).$
Before doing so, let us first look at the consequences a violation of the $k$-increasingness for $C$ would have. 

Without loss of generality we may assume that the $k$-increasingness is violated for a $k$-box $\skb\in\R_k(\D)$, i.e., 
$\Lk{C,C}(\skb)<0$.
Note that the $k$-box $\skb$ has exactly $2^k$ vertices, half of them with positive multiplicities. We shall denote these vertices by $\x_i$, i.e., for each $i\in[2^{k-1}]$
we have
\begin{equation}\label{eq:step4A:defXis}
    \x_i\in\ver \skb  \qquad\text{and}\qquad m_\skb(\x_i)=1. 
\end{equation}
The following lemma illustrates that for a  subset thereof the values of $B$ and $C$ differ and give rise to the existence of a disjoint union of $k$-boxes with respect to which the vertex has a negative multiplicity.

\begin{lemma}\label{le:step4A:existenceRi}
Let $\D$ be a dense countably infinite mesh in $\II^n$ and fix some integer $k\in[n]$.  
Let $B,C \colon \D \to \II$ be functions with $C \le B$, $L_k^{(C,B)}(\dub)\ge 0$ for all $\dub\in \R_k(\D)$, and $\gkD^{(C,B)}(\dd)=0$ for all $\dd\in\D$. 
Furthermore, assume that $C(\vv) = B(\vv)$ for all vertices $\vv$ of the unit cube $\II^n$ and that there exists a $k$-box $\skb\in\R_k(\D)$ with
\begin{equation}\label{eq:step4A:kincrC}
    \Lk{C,C}(\skb)=v< 0.
\end{equation}
Then there exists $s\in[2^{k-1}]$   
such that for each $i\in[s]$ 
there exist a vertex $\x_i\in\ver \skb$ and a finite disjoint union of $k$-boxes $\dub_i\in\R_k(\D)$ with 
\begin{equation}\label{eq:step4A:propXis}
C(\x_i)<B(\x_i), \qquad m_{\skb}(\x_i)=1, \qquad m_{\dub_i}(\x_i)<0\qquad \text{and}\qquad \frac{\LkCB(\dub_i)}{|m_{\dub_i}(\x_i)|}<\frac{|v|}{s}.
\end{equation}
For all other $\x_i \in \ver \skb$ with $m_{\skb}(\x_i)=1$ it holds that $C(\x_i) = B(\x_i)$.
\end{lemma}

\begin{proof}
Let $\skb\in\R_k(\D)$ 
be a $k$-box with $\Lk{C,C}(\skb)=v<0$,
and denote by $\x_i$ its vertices with positive multiplicity, i.e., for each $i\in[2^{k-1}]$ we have
$\x_i\in\ver \skb$ and $m_{\skb}(\x_i)=1$  (as in~\eqref{eq:step4A:defXis}).

Since $\gkD^{(C,B)}(\dd)=\min\{\POkD^{(C,B)}(\dd), B(\dd)-C(\dd)\}=0$ for all $\dd\in\D$, this holds in particular also for all $\x_i$. Assuming that $B(\x_i)-C(\x_i)=0$ for all vertices $\x_i$ with $i\in[2^{k-1}]$,
leads, on the basis of the assumptions for $C$ (compare  also Proposition~\ref{pr:step3A}~(iii)), to the contradiction
\begin{equation*}
0\le \Lk{C,B}(\skb)=\Lk{C,C}(\skb)<0.
\end{equation*}  
 We may therefore assume that (after a possible rearrangement of the vertices) there exists $s\in[2^{k-1}]$ such that
\begin{equation}\label{eq:gamma_CB_kD}
    \gkD^{(C,B)}(\x_i)=
\begin{cases}
\POkDCB(\x_i)=0<B(\x_i)-C(\x_i) & \text{if } i\in[ s],\\
B(\x_i)-C(\x_i)=0 & \text{if } i\in [2^{k-1}]\setminus [s].
\end{cases}
\end{equation}
Since for each $i\in[s]$ we have $B(\x_i) > C(\x_i)$, the point $\x_i$ is not a vertex of the unit cube $\II^n$, and by Lemma~\ref{le:KBoxesWithMRFromZ2} there exists an $\dub'_i\in\R_k(\D)$ with $m_{\dub'_i}(\x_i) <0$. Since, for all $\dd\in\D$, the function $\POkDCB$ is given by 
\begin{equation*}
    \POkDCB(\dd)=\inf_{\substack{\dub\in\R_k(\D),\\ m_\dub(\dd)<0}} \frac{\LkCB(\dub)}{|m_\dub(\dd)|},
\end{equation*}
we can further conclude that, for each $i\in[s]$, the infimum for $\POkDCB(\x_i)$ is not taken over the empty set, implying that there exists  an $\dub_i\in\R_k(\D)$   with
\begin{equation*}
    m_{\dub_i}(\x_i)<0 \qquad\text{and}\qquad \frac{\LkCB(\dub_i)}{|m_{\dub_i}(\x_i)|}<\frac{|v|}{s}
\end{equation*}
while $m_{\skb}(\x_i)=1$ for $i\in[s]$ is trivially fulfilled.
\end{proof}

Under the assumptions of Lemma~\ref{le:step4A:existenceRi} we obtain $\dub_i\in\R_k(\D)$ with $i\in[s]$ for some
$s\in[2^{k-1}]$, from which additional finite disjoint unions of $k$-boxes  $\widehat{\dub}$ and $\widehat{\dub}_i$ can be constructed. Taking into account that each corresponding vertex $\x_i\in\ver \skb$ fulfills \eqref{eq:step4A:propXis}, and in particular also $C(\x_i)<B(\x_i)$ for all $i\in[s]$, we define the following additional disjoint unions of $k$-boxes putting, for $i\in[s]$,
\begin{align}
m & =|m_{\dub_1}(\x_1)|\cdot\ldots\cdot|m_{\dub_s}(\x_s)|
\qquad \text{and} \qquad
    m_i =\frac{m}{|m_{\dub_i}(\x_i)|},\notag\\[1ex]
    \widehat{\dub} &=\left(\bigsqcup_{t=1}^m \skb\right)\sqcup\left(\bigsqcup_{t=1}^{m_1} \dub_1\right)\sqcup\dots \sqcup \left(\bigsqcup_{t=1}^{m_s} \dub_s\right)=\left(\bigsqcup_{t=1}^m \skb\right)\sqcup\left(\bigsqcup_{l=1}^{s} \left(\bigsqcup_{t=1}^{m_l} \dub_l\right)\right)
    \label{eq:step4A:DefR0},\\
    \widehat{\dub}_i &=\left(\bigsqcup_{t=1}^m \skb\right)\sqcup\left(\bigsqcup_{\substack{l=1\\ l\neq i}}^{s} \left(\bigsqcup_{t=1}^{m_l} \dub_l\right)\right). \label{eq:step4A:DefRi}
\end{align}
For every $\dd\in\D$, the multiplicities with respect to $\widehat{\dub}$ and $\widehat{\dub}_i$, for $i\in[s]$, can be evaluated as
\begin{equation}
    m_{\widehat{\dub}}(\dd)=m\cdot m_\skb(\dd)+\sum_{l=1}^s m_l\cdot m_{\dub_l}(\dd) \qquad \text{and} \qquad
    m_{\widehat{\dub}_i}(\dd)=m\cdot m_\skb(\dd)+\sum_{l\in[s]\setminus\{i\}} m_l\cdot m_{\dub_l}(\dd). \label{eq:mRhati}
\end{equation}
In particular, for each vertex $\x_i$ of the $k$-box $\skb$ with $m_{\skb}(\x_i)=1$ and $i\in[s]$ we obtain
\begin{align*}
    m_{\widehat{\dub}}(\x_i)&=m\cdot \underbrace{m_\skb(\x_i)}_{=1}+\sum_{l=1}^s m_l\cdot m_{\dub_l}(\x_i) \notag
    \\&
        = m + \underbrace{m_i\cdot \underbrace{m_{\dub_i}(\x_i)}_{<0}}_{=-m} 
            +\sum_{l\in[s]\setminus\{i\}} m_l\cdot m_{\dub_l}(\x_i)=\sum_{l\in[s]\setminus\{i\}} m_l\cdot m_{\dub_l}(\x_i)
\end{align*}
and
\begin{equation}\label{eq:step4A:mRi(xi)}
    m_{\widehat{\dub}_i}(\x_i) =m\cdot m_\skb(\x_i)+\sum_{l\in[s]\setminus\{i\}} m_l\cdot m_{\dub_l}(\x_i)=m+m_{\widehat{\dub}}(\x_i).
\end{equation}

Having disjoint unions of $k$-boxes $\widehat{\dub}$ and $\widehat{\dub}_i$ for $i\in[s]$  at hand, we are now interested in identifying upper bounds for $\LkCB(\widehat{\dub})$ and $\LkCB(\widehat{\dub}_i)$.  

\begin{proposition}\label{pr:step4A:upperboundsLkRhats}
Let $\D$ be a dense countably infinite mesh in $\II^n$ and fix some integer $k\in[n]$. 
Let $B,C \colon \D \to \II$ be functions with $C \le B$, $L_k^{(C,B)}(\dub)\ge 0$ for all $\dub\in \R_k(\D)$, and $\gkD^{(C,B)}(\dd)=0$ for all $\dd\in\D$.
Furthermore, assume that $C(\vv) = B(\vv)$ for all vertices $\vv$ of the unit cube $\II^n$ and that there exists a $k$-box $\skb\in\R_k(\D)$ with
$\Lk{C,C}(\skb)=v<0$ such that there are vertices $\x_i$ with $m_{\skb}(\x_i)=1$ and $C(\x_i)<B(\x_i)$ for all $i\in[s]$
and some $s\in[2^{k-1}]$,
while for all other $\x_i \in \ver \skb$ with $m_{\skb}(\x_i)=1$ the equality $C(\x_i) = B(\x_i)$ holds.
\begin{itemize} 
\item[\textup{(i)}] For each $i\in[s]$ and each $\widehat{\dub}_i$ 
as defined by~\eqref{eq:step4A:DefRi} we obtain
\begin{equation}\label{eq:step4A:upperboundsLkRhats:Ri}
    \LkCB(\widehat{\dub}_i)\le m\cdot \big(B(\x_i)-C(\x_i)\big)+m\cdot \Lk{C,C}(\skb)
        +\sum_{l\in[s]\setminus\{i\}} m_l\cdot \LkCB(\dub_l).
\end{equation}
\item[\textup{(ii)}] For $\widehat{\dub}$ as defined by~\eqref{eq:step4A:DefR0} we get
\begin{equation}\label{eq:step4A:upperboundsLkRhats:R0}
    \LkCB(\widehat{\dub})\le m\cdot \Lk{C,C}(\skb)+\sum_{l=1}^s m_l\cdot \LkCB(\dub_l).
\end{equation}
\end{itemize}
\end{proposition}

\begin{proof}
We first focus on the upper bound for $\LkCB(\widehat{\dub}_i)$ for some arbitrary but fixed $i\in[s]$.
Due to $C\le B$, and taking into account~\eqref{eq:LKonD}, we can express 
\begin{equation*}
\LkCB(\widehat{\dub}_i)
    = \sum_{\dd\in\D}\max\{m_{\widehat{\dub}_i}(\dd)B(\dd), m_{\widehat{\dub}_i}(\dd)C(\dd)\},
\end{equation*}
leading to the following equivalent expression of~\eqref{eq:step4A:upperboundsLkRhats:Ri} 
\begin{align*}
    \sum_{\dd\in\D}\max\{m_{\widehat{\dub}_i}(\dd)&\cdot B(\dd), m_{\widehat{\dub}_i}(\dd)\cdot C(\dd)\} 
       \le m\cdot \left(B(\x_i)-C(\x_i)\right)+m\cdot \sum_{\dd\in\D}m_{\skb}(\dd)\cdot C(\dd)\\
        &+\sum_{l\in[s]\setminus\{i\}} m_l\cdot 
            \left(\sum_{\dd\in\D}\max\{m_{\dub_l}(\dd)\cdot B(\dd), m_{\dub_l}(\dd)\cdot C(\dd)\}\right).
\end{align*}
We shall investigate the contribution of each $\dd\in\D$ to both sides of the above equivalent form of inequality~\eqref{eq:step4A:upperboundsLkRhats:Ri} by distinguishing the following three cases: (1) $\dd=\x_i$, (2) $\dd=\x_j$, with $j\in[s]\setminus\{i\}$ and (3) $\dd\in\D\setminus\{\x_j\mid j\in[s]\}$, i.e., whether or not $\dd$ is one of the vertices of $\skb$ with positive multiplicity and different values at $B$ and $C$.

\emph{Case} $1$: Suppose $\dd=\x_i$, i.e., $m_\skb(\x_i)=1$ and $B(\x_i)>C(\x_i)$ (and $\POkDCB(\x_i)=0$). 
We consider two subcases. First assume that $m_{\widehat{\dub}_i}(\x_i)\ge 0$. By \eqref{eq:step4A:mRi(xi)}, for the contribution of $\x_i$ to $\LkCB(\widehat{\dub}_i)$ we obtain the following (in)equalities: 
    \begin{align*}
         m_{\widehat{\dub}_i}(\x_i)\cdot B(\x_i)
         &=m\cdot m_\skb(\x_i) \cdot B(\x_i)+\sum_{l\in[s]\setminus\{i\}} m_l\cdot m_{\dub_l}(\x_i)\cdot B(\x_i)\\
         &=m \cdot B(\x_i)+\sum_{l\in[s]\setminus\{i\}} m_l\cdot m_{\dub_l}(\x_i)\cdot B(\x_i)\\
        &= m (B(\x_i)-C(\x_i))+m\cdot m_\skb(\x_i)\cdot C(\x_i)+\sum_{l\in[s]\setminus\{i\}} m_l\cdot m_{\dub_l}(\x_i)\cdot B(\x_i)\\
        &\le m (B(\x_i)-C(\x_i))+m\cdot m_\skb(\x_i)\cdot C(\x_i)\\[1ex]
        &\quad{}+\sum_{l\in[s]\setminus\{i\}} m_l\cdot \max\{m_{\dub_l}(\x_i)\cdot B(\x_i),m_{\dub_l}(\x_i)\cdot C(\x_i)\}.
    \end{align*}
     Secondly, if $m_{\widehat{\dub}_i}(\x_i)< 0$, then for the contribution of $\x_i$ to $\LkCB(\widehat{\dub}_i)$ we obtain, taking into account $C(\x_i)< B(\x_i)$, 
    \begin{align*}
         m_{\widehat{\dub}_i}(\x_i)\cdot C(\x_i)
         &=m\cdot m_\skb(\x_i) \cdot C(\x_i)+\sum_{l\in[s]\setminus\{i\}} m_l\cdot m_{\dub_l}(\x_i) \cdot C(\x_i)\\
        &\le m\cdot  (B(\x_i)-C(\x_i))+m\cdot m_\skb(\x_i)\cdot C(\x_i)
        \\[1ex]
        &\quad{} 
        +\sum_{l\in[s]\setminus\{i\}} m_l\cdot \max\{m_{\dub_l}(\x_i)\cdot B(\x_i),m_{\dub_l}(\x_i)\cdot C(\x_i)\}.
    \end{align*}
    
\emph{Case} $2$: Consider $\dd=\x_j$ with $j\in[s]\setminus\{i\}$, i.e., $m_\skb(\x_j)=1$, and $C(\x_j)< B(\x_j)$. Assuming first that $m_{\widehat{\dub}_i}(\x_j)\ge 0$ we obtain for the contribution of $\x_j$ to $\LkCB(\widehat{\dub}_i)$ the following series of (in)equalities:  
        \begin{align*}
        m_{\widehat{\dub}_i}(\x_j)\cdot B(\x_j)&=m\cdot m_\skb(\x_j)\cdot B(\x_j)
                +\sum_{l\in[s]\setminus\{i\}} m_l\cdot m_{\dub_l}(\x_j)\cdot B(\x_j)\\
        &=m\cdot B(\x_j)+\underbrace{m_j\cdot m_{\dub_j}(\x_j)}_{=-m}\cdot B(\x_j)
                +\sum_{l\in[s]\setminus\{i,j\}} m_l\cdot m_{\dub_l}(\x_j) \cdot B(\x_j)\\
        &=\sum_{l\in[s]\setminus\{i,j\}} m_l\cdot m_{\dub_l}(\x_j)\cdot B(\x_j)\\
        &\le \sum_{l\in[s]\setminus\{i,j\}} m_l\cdot \max\{m_{\dub_l}(\x_j)\cdot B(\x_j), m_{\dub_l}(\x_j) \cdot C(\x_j)\}\\
        &= m\cdot m_\skb(\x_j)\cdot C(\x_j)-m \cdot C(\x_j)\\[1ex]
            &\quad{} +\sum_{l\in[s]\setminus\{i,j\}} m_l\cdot \max\{m_{\dub_l}(\x_j)\cdot B(\x_j), m_{\dub_l}(\x_j) \cdot C(\x_j)\}.
        \end{align*}
    Using the definition of $m$, it follows that
    \begin{align*}
       m_{\widehat{\dub}_i}(\x_j)\cdot B(\x_j) &\le m\cdot m_\skb(\x_j)\cdot C(\x_j)+m_j\cdot m_{\dub_j}(\x_j) \cdot C(\x_j)\\
            &\quad{} +\sum_{l\in[s]\setminus\{i,j\}} m_l\cdot \max\{m_{\dub_l}(\x_j)\cdot B(\x_j), m_{\dub_l}(\x_j) \cdot C(\x_j)\}\\
        &\le  m\cdot m_\skb(\x_j)\cdot C(\x_j) \\
        &\quad{} +\sum_{l\in[s]\setminus\{i\}} 
                m_l\cdot \max\{m_{\dub_l}(\x_j)\cdot B(\x_j), m_{\dub_l}(\x_j) \cdot C(\x_j)\}.
    \end{align*}
    Next assume that $m_{\widehat{\dub}_i}(\x_j)< 0$. Then the contribution of $\x_j$ to $\LkCB(\widehat{\dub}_i)$ satisfies 
    \begin{align*}
        m_{\widehat{\dub}_i}(\x_j)\cdot C(\x_j)&=m\cdot m_\skb(\x_j) \cdot C(\x_j)
                +\sum_{l\in[s]\setminus\{i\}} m_l\cdot m_{\dub_l}(\x_j)\cdot C(\x_j)\\
        &\le  m\cdot m_\skb(\x_j)\cdot C(\x_j) \\ 
        &\quad{} +\sum_{l\in[s]\setminus\{i\}} 
                m_l\cdot \max\{m_{\dub_l}(\x_j)\cdot B(\x_j), m_{\dub_l}(\x_j) \cdot C(\x_j)\}.
    \end{align*}

\emph{Case} $3$: 
Consider $\dd\in\D\setminus\{\x_1,\dots, \x_s\}$.
If $m_\skb(\dd)=1$ then $\dd$ is a vertex of $\skb$ with positive multiplicity and necessarily fulfilling $B(\dd)=C(\dd)$. As a consequence,
\begin{align*}
    \max\{m_{\widehat{\dub}_i}(\dd)&\cdot B(\dd), m_{\widehat{\dub}_i}(\dd)\cdot C(\dd)\}\\
    &=m_{\widehat{\dub}_i}(\dd)\cdot C(\dd)\\
    &=m\cdot m_{\skb}(\dd) \cdot C(\dd) + \sum_{l\in[s]\setminus\{i\}} m_l \cdot m_{\dub_l}(\dd)\cdot C(\dd)\\
    &=m\cdot m_{\skb}(\dd) \cdot C(\dd) 
        + \sum_{l\in[s]\setminus\{i\}} m_l \cdot \max\{m_{\dub_l}(\dd)\cdot B(\dd), m_{\dub_l}(\dd)\cdot C(\dd)\}.
\end{align*}
Otherwise $m_{\skb}(\dd)\in\{-1,0\}$, i.e., $m\cdot m_\skb(\dd) \cdot B(\dd)\le m\cdot m_\skb(\dd) \cdot C(\dd)$. Therefore, evaluating each term separately using \eqref{eq:mRhati}, we obtain
    \begin{align*}
    \max\{&m_{\widehat{\dub}_i}(\dd)\cdot B(\dd), m_{\widehat{\dub}_i}(\dd)\cdot C(\dd)\}\\[1ex]
        &\le \max\{m\cdot m_{\skb}(\dd)\cdot B(\dd), m\cdot m_{\skb}(\dd)\cdot C(\dd)\} \\
        &\quad{} + \sum_{l\in[s]\setminus\{i\}} m_l\cdot \max\{m_{\dub_l}(\dd)\cdot B(\dd),m_{\dub_l}(\dd)\cdot C(\dd)\}\\
        &=m\cdot m_{\skb}(\dd)\cdot C(\dd)+ \sum_{l\in[s]\setminus\{i\}} m_l\cdot \max\{m_{\dub_l}(\dd)\cdot B(\dd),m_{\dub_l}(\dd)\cdot C(\dd)\}.
        \end{align*}
Summarizing all three cases we have shown that~\eqref{eq:step4A:upperboundsLkRhats:Ri} holds since the contribution of any $\dd\in D$ to its left-hand side is surpassed by its contribution to the right-hand side of the inequality.

We now turn to $\LkCB(\widehat{\dub})$ and show that the inequality~\eqref{eq:step4A:upperboundsLkRhats:R0} holds  which, following~\eqref{eq:LKonD}, can be equivalently expressed by
\begin{equation*}
\begin{split}
    \sum_{\dd\in\D}\max&\{m_{\widehat{\dub}}(\dd) \cdot B(\dd), m_{\widehat{\dub}}(\dd) \cdot C(\dd)\}\\
        &\le \sum_{\dd\in\D} m\cdot m_{\skb}(\dd) \cdot C(\dd) + \sum_{l=1}^s m_l\cdot \left(\sum_{\dd\in\D}\max\{m_{\dub_l}(\dd)\cdot B(\dd), m_{\dub_l}(\dd)\cdot C(\dd) \}\right).
\end{split}
\end{equation*}
We again look at the contribution of each $\dd\in\D$ to both sides of this inequality and distinguish different cases for $\dd\in\D$.

\emph{Case} $1$: $\dd\in\{\x_1,\dots, \x_s\}$, i.e., $m_{\skb}(\x_j)=1$ and $C(\x_j)<B(\x_j)$ for each $j\in[s]$.
Assuming that $m_{\widehat{\dub}}(\x_j)>0$,  \eqref{eq:step4A:mRi(xi)} implies that also $m_{\widehat{\dub}_j}(\x_j)=m+m_{\widehat{\dub}}(\x_j)>0$. As a consequence,
\begin{equation*}\PMkD^{(C,B)}(\x_j)=\inf_{\substack{\dub\in\R_k(\D)\\ m_\dub(\x_j)>0}}\frac{\Lk{C,B}(\dub)}{|m_\dub(\x_j)|}\le \frac{\Lk{C,B}(\widehat{\dub}_j)}{|m_{\widehat{\dub}_j}(\x_j)|}.
\end{equation*}
Since $\POkDCB(\x_j)=0$, it follows from Proposition~\ref{pr:step2} that 
\begin{equation*}
    |m_{\widehat{\dub}_j}(\x_j)|\cdot(B(\x_j)-C(\x_j))\le |m_{\widehat{\dub}_j}(\x_j)|\cdot(\POkDCB(\x_j)+\PMkD^{(C,B)}(\x_j))\le \LkCB(\widehat{\dub}_j).
\end{equation*}
Moreover, by~\eqref{eq:step4A:upperboundsLkRhats:Ri} and taking into account~\eqref{eq:step4A:kincrC} and~\eqref{eq:step4A:propXis}, we may argue that
\begin{align*}
    m_{\widehat{\dub}_j}(\x_j)\cdot(B(\x_j)-C(\x_j))&\le \LkCB(\widehat{\dub}_j)\\
        &\le m\cdot (B(\x_j)-C(\x_j))+m\cdot \Lk{C,C}(\skb)+\sum_{l\in[s]\setminus\{j\}}m_l\cdot \LkCB(\dub_l)\\
        &<m\cdot (B(\x_j)-C(\x_j))+m\cdot v+\sum_{l\in[s]\setminus\{j\}} m_l\cdot |m_{\dub_l}(\x_l)|\cdot \frac{|v|}{s} \\
        &=m\cdot (B(\x_j)-C(\x_j))+m\cdot v+ m \frac{|v|}{s} (s-1)\\
        &=m\cdot (B(\x_j)-C(\x_j))- \frac{m|v|}{s}.
\end{align*}
Since $m_{\widehat{\dub}_j}(\x_j)=m+m_{\widehat{\dub}}(\x_j)$, we further obtain the contradiction
\begin{align*}
    0\le \underbrace{m_{\widehat{\dub}}(\x_j)}_{>0}\cdot (B(\x_j)-C(\x_j))< -\frac{m\cdot |v|}{s}<0,
\end{align*}
showing that necessarily $m_{\widehat{\dub}}(\x_j)\le 0$ for all $j\in[s]$.
Therefore, for the contribution of any $\x_j$ with $j\in[s]$
to the left-hand side of~\eqref{eq:step4A:upperboundsLkRhats:R0} we obtain 
\begin{align*}
    m_{\widehat{\dub}}(\x_j)\cdot C(\x_j)&=\left(m\cdot m_{\skb}(\x_j)+\sum_{l=1}^s m_l\cdot m_{\dub_l}(\x_j)\right)\cdot C(\x_j)\\
        &\le m\cdot m_\skb(\x_j)\cdot C(\x_j)+\sum_{l=1}^s m_l \cdot\max\{m_{\dub_l}(\x_j)\cdot B(\x_j),m_{\dub_l}(\x_j)\cdot C(\x_j)\}.
\end{align*}

\emph{Case} $2$:  $\dd\in\D\setminus\{\x_1,\dots, \x_s\}$. If $m_{\skb}(\dd)=1$ then $\dd$ is a vertex of~$\skb$ with positive multiplicity and necessarily fulfills $B(\dd)=C(\dd)$ (compare also~\eqref{eq:gamma_CB_kD}). Then $m_{\widehat{\dub}}(\dd) \cdot B(\dd)=m_{\widehat{\dub}}(\dd) \cdot C(\dd)$ and $m_{\dub_l}(\dd) \cdot B(\dd)=m_{\dub_l}(\dd) \cdot C(\dd)$ for all $l\in[s]$, 
and therefore trivially
\begin{align*}
    \max\{m_{\widehat{\dub}}(\dd) \cdot B(\dd)&, m_{\widehat{\dub}}(\dd) \cdot C(\dd)\}=m_{\widehat{\dub}}(\dd) \cdot C(\dd)\\
        &=m\cdot m_{\skb}(\dd) \cdot C(\dd) + \sum_{l=1}^s m_l \cdot m_{\dub_l}(\dd)\cdot C(\dd)\\
        &=m\cdot m_{\skb}(\dd) \cdot C(\dd) + \sum_{l=1}^s m_l \cdot \max\{m_{\dub_l}(\dd)\cdot B(\dd),m_{\dub_l}(\dd)\cdot C(\dd)\}.
\end{align*}
If  $m_{\skb}(\dd)\neq 1$ then $\dd$ is not a vertex of $\skb$ or it is a vertex with negative multiplicity to $\skb$, i.e., fulfills $m_{\skb}(\dd)\in\{0,-1\}$ since $\skb$ is a $k$-box.
Assume first that $m_{\widehat{\dub}}(\dd)\ge 0$. Then, due to $m_{\skb}(\dd)\in\{0,-1\}$, the contribution to the left-hand side of ~\eqref{eq:step4A:upperboundsLkRhats:R0} satisfies  
\begin{align*}
    m_{\widehat{\dub}}(\dd)\cdot B(\dd)&= m\cdot m_{\skb}(\dd)\cdot B(\dd) + \sum_{l=1}^s m_l\cdot m_{\dub_l}(\dd) \cdot B(\dd)\\
    &\le m\cdot m_{\skb}(\dd)\cdot C(\dd) + \sum_{l=1}^s m_l\cdot m_{\dub_l}(\dd \cdot B(\dd))\\
    &\le m\cdot m_{\skb}(\dd)\cdot C(\dd) + \sum_{l=1}^s m_l\cdot \max\{m_{\dub_l}(\dd) \cdot B(\dd),m_{\dub_l}(\dd) \cdot C(\dd)\}.
\end{align*}
If $m_{\widehat{\dub}}(\dd)< 0$ then for the contribution to the left-hand side of ~\eqref{eq:step4A:upperboundsLkRhats:R0} we obtain 
\begin{align*}
    m_{\widehat{\dub}}(\dd)\cdot C(\dd)&= m\cdot m_{\skb}(\dd)\cdot C(\dd) + \sum_{l=1}^s m_l\cdot m_{\dub_l}(\dd) \cdot C(\dd)\\
    &\le m\cdot m_{\skb}(\dd)\cdot C(\dd) + \sum_{l=1}^s m_l\cdot \max\{m_{\dub_l}(\dd) \cdot B(\dd),m_{\dub_l}(\dd) \cdot C(\dd)\},
\end{align*}
verifying that inequality~\eqref{eq:step4A:upperboundsLkRhats:R0} holds. 
\end{proof}

We have now collected all the necessary details for showing that the function $C$ obtained as the pointwise limit of the sequence $(A^{(i)})_{i\in\NN}$ does not only have all the properties shown in Proposition~\ref{pr:step3A} but is also $k$-increasing on $\D$. 

\begin{proposition}\label{pr:step4A}
Let $\D$ be a dense countably infinite mesh in $\II^n$ and fix some integer $k\in[n]$. 
Let $A, B \colon \D \to \II$ be functions with $A \le B$ and $L_k^{(A,B)}(\dub)\ge 0$ for all $\dub\in \R_k(\D)$. Furthermore, assume that $A(\vv) = B(\vv)$ for all vertices $\vv$ of the unit cube $\II^n$. 
Let $C\colon\D\to\II$ be the pointwise limit of the sequence $(A^{(i)})_{i\in\NN}$ defined in~\eqref{eq:step3A:DefAis}.
Then $C$ is $k$-increasing on $\D$, i.e., $\Lk{C,C}(\dub)\ge 0 \text{ for all } \dub\in\R_k(\D)$.
\end{proposition}

\begin{proof}
Note that $C=\lim_{i\to\infty} A^{(i)}$ implies that for each $\dd\in\D$
\begin{equation*}
\gkD^{(C,B)}(\dd)=\min\{\POkD^{(C,B)}(\dd), B(\dd)-C(\dd)\}=0,   
\end{equation*}
and $\LkCB(\dub)\ge 0 $ for all $\dub\in\R_k(\D)$, according to Proposition~\ref{pr:step3A}. 

Assume that $C$ is not $k$-increasing, i.e., there exists a $k$-box $\skb$ such that 
$\Lk{C,C}(\skb)=v<0$,
where the vertices $\x_i$ of $\skb$ fulfill $m_{\skb}(\x_i)=1$ and $C(\x_i)<B(\x_i)$ for all $i\in[s]$
and some $s\in[2^{k-1}]$,
while $C(\x_i)=B(\x_i)$ for all other vertices with $m_{\skb}(\x_i)=1$. 
Note that $C(\vv) = B(\vv)$ for all vertices $\vv$ of the unit cube $\II^n$ since $A \le C \le B$ and $A(\vv) = B(\vv)$.
Then, following Lemma~\ref{le:step4A:existenceRi}, there exist finite disjoint unions of $k$-boxes $\dub_i\in\R_k(\D)$ such that for all
$i\in[s]$
\begin{equation*}
    m_{\dub_i}(\x_i)<0\qquad \text{and}\qquad \frac{\LkCB(\dub_i)}{|m_{\dub_i}(\x_i)|}<\frac{|v|}{s}.
\end{equation*}
Combining the $k$-box $\skb$ and the corresponding disjoint unions of $k$-boxes $\dub_i$, we introduce an additional disjoint union $\widehat{\dub}\in\R_k(\D)$ by means of~\eqref{eq:step4A:DefR0}. 
Proposition~\ref{pr:step4A:upperboundsLkRhats} provides us with an upper bound for $\Lk{C,B}(\widehat{\dub})$ by means of~\eqref{eq:step4A:upperboundsLkRhats:R0}, implying the contradiction
\begin{align*}
    0\le \LkCB(\widehat{\dub})&\le m\cdot \Lk{C,C}(\skb)+\sum_{l=1}^s m_l \cdot \LkCB(\dub_l)\\
    &<m\cdot v + \sum_{l=1}^s m_l \cdot \frac{|v|}{s}\cdot |m_{\dub_l}(\x_l)|\\
    &=m\cdot (v+|v|)=0,
\end{align*}
showing that $C$ has to be $k$-increasing on $\D$. 
\end{proof}

\section{Construction of $C$ from above and discussion of its properties}\label{se:CbyloweringB_short}

The results of the previous section already prove the existence of a $k$-increasing $n$-variate function~$C$ on a dense countable mesh $\D$ by a construction from below, i.e., starting from the lower bound~$A$. A rather natural question is whether or not the construction of a, possibly different, function $C$ could also be initiated from the upper bound~$B$. This question can be answered to the positive. In this section we briefly sketch the proof steps and the construction, pointing to possibly different arguments needed in the proofs in comparison to the results related to the construction of a function~$C$ from the lower bound~$A$. 

In the single construction step, function $B$ defined on the mesh $\D$ is reduced at an arbitrary but fixed point $\x\in\D$ by means of $\dkD^{(A,B)}$, leading to a smaller function $B'$ still satisfying $\Lk{A,B'}(\dub)\ge 0$ for all $\dub\in\R_k(D)$. The proof of the following proposition formalizing this step is in complete analogy to the proof of Proposition~\ref{pr:step1A}.

\begin{proposition}\label{pr:step1B}
Let $\D$ be a dense countably infinite mesh in $\II^n$ and fix some integer $k\in[n]$.  
Let $A, B\colon \D \to \II$ be functions with $A \le B$  and $L_k^{(A,B)}(\dub)\ge 0$ for all $\dub\in \R_k(\D)$.  Fix a point $\x \in \D$ and define the function $B'\colon \D \to \II$ by
\begin{equation*}
B'(\uu) =  \begin{cases}
            B(\uu) & \text{if }   \uu \ne \x, \\
            B(\x) - \dkD^{(A,B)}(\x) & \text{if }   \uu = \x.
            \end{cases}    
\end{equation*}
Then we have that $A\le B' \le B$, that the pair $(A, B')$ satisfies the condition $L_k^{(A,B')}(\dub)\ge 0$ for all $\dub\in \R_k(\D)$, and that $\dkD^{(A,B')}(\x) = 0.$
\end{proposition}

Since $\D$ is a dense countably infinite mesh in $\II^n$ containing the points $\mathbf{0}$ and $\1$, its elements can be rearranged into a sequence $(\dd_i)_{i\in\NN}$. As a consequence, a function $C$ can be obtained as the pointwise limit of a sequence of functions $B^{(i)}\colon\D\to\II$ recursively defined by successively applying Proposition~\ref{pr:step1B}, i.e., putting $B^{(0)}= B$ and, for $i\geq 1$,
\begin{equation}
\label{eq:step3B:DefBis}
B^{(i)}(\dd) =  \begin{cases}
            B^{(i-1)}(\dd)  & \text{if }  \dd \ne \dd_i, \\
            B^{(i-1)}(\dd_i) - \dkD^{(A,B^{(i-1)})}(\dd_i)  & \text{if }  \dd = \dd_i,
            \end{cases}
\end{equation}
(compare also Proposition~\ref{pr:step3A}). The function $C$  constructed in this way has the following properties.
\begin{proposition}\label{pr:step3B}
Let $\D$ be a dense countably infinite mesh in $\II^n$ and fix some integer $k\in[n]$.
Let $A, B \colon \D \to \II$ be functions with $A \le B$ and  $L_k^{(A,B)}(\dub)\ge 0$ for all $\dub\in \R_k(\D)$. Then there exists $C\colon\D\to\II$ such that
\begin{itemize}
	\item[\textup{(i)}] $A \le C \le B$ on $\D$,
	\item[\textup{(ii)}] $\dkD^{(A,C)}(\dd) = 0$ for all $\dd \in \D$,
	\item[\textup{(iii)}] $\LkAC(\dub) \ge 0$ for all $\dub\in\R_k(\D)$.
\end{itemize}
\end{proposition}

When showing that $C$ is $k$-increasing, assume first, to the contrary, that there is some $k$-box $\skb$ where the vertices $\x_i\in\ver \skb$ with \emph{negative} multiplicities, i.e., fulfilling $m_\skb(\x_i)=-1$, are in the focus of our argumentation, in particular those fulfilling in addition $A(\x_i)< C(\x_i)$ with $i\in[s]$ and some 
$s\in[2^{k-1}]$.

\begin{lemma}\label{le:step4B:existenceRi}
Let $\D$ be a dense countably infinite mesh in $\II^n$ and fix some integer $k\in[n]$.
Let $A,C \colon \D \to \II$ be functions with $A\le C $, $L_k^{(A,C)}(\dub)\ge 0$ for all $\dub\in \R_k(\D)$, and $\dkD^{(A,C)}(\dd)=0$ for all $\dd\in\D$.
Furthermore, assume that $A(\vv) = C(\vv)$ for all vertices $\vv$ of the unit cube~$\II^n$ and that there exists a $k$-box $\skb\in\R_k(\D)$ with
$\Lk{C,C}(\skb)=v< 0$.
Then there exists $s\in[2^{k-1}]$
such that for each $i\in[s]$
there exist a vertex $\x_i\in\ver \skb$ and a finite disjoint union of $k$-boxes $\dub_i\in\R_k(\D)$ with 
\begin{equation*}
A(\x_i) < C(\x_i), \qquad m_{\skb}(\x_i)=-1, \qquad m_{\dub_i}(\x_i)>0\qquad \text{and}\qquad \frac{\LkAC(\dub_i)}{|m_{\dub_i}(\x_i)|}<\frac{|v|}{s}.
\end{equation*}
For all other $\x_i \in \ver \skb$ with $m_{\skb}(\x_i)=-1$ the equality $A(\x_i) = C(\x_i)$ holds. 
\end{lemma}
The proof of Lemma~\ref{le:step4B:existenceRi} follows in analogy to the arguments for Lemma~\ref{le:step4A:existenceRi} and is thus omitted at this place. 

For any such vertex $\x_i$ with $i\in[s]$, $m_\skb(\x_i)=-1$ and $m_{\dub_i}(\x_i)>0$ the following notations and additional disjoint unions of $k$-boxes can be defined by putting, for $i\in[s]$,
\begin{align}
    m & =|m_{\dub_1}(\x_1)|\cdot\ldots\cdot|m_{\dub_s}(\x_s)| \qquad \text{and} \qquad
    m_i =\frac{m}{|m_{\dub_i}(\x_i)|},\notag\\[1ex]
    \widecheck{\dub} & =\left(\bigsqcup_{t=1}^m \skb\right)\sqcup\left(\bigsqcup_{t=1}^{m_1} \dub_1\right)\sqcup\dots \sqcup \left(\bigsqcup_{t=1}^{m_s} \dub_s\right)=\left(\bigsqcup_{t=1}^m \skb\right)\sqcup\left(\bigsqcup_{l=1}^{s} \left(\bigsqcup_{t=1}^{m_l} \dub_l\right)\right)\label{eq:step4B:DefR0},\\
    \widecheck{\dub}_i & =\left(\bigsqcup_{t=1}^m \skb\right)\sqcup\left(\bigsqcup_{\substack{l=1\\ l\neq i}}^{s} \left(\bigsqcup_{t=1}^{m_l} \dub_l\right)\right).\label{eq:step4B:DefRi}
\end{align}
In analogy to Proposition~\ref{pr:step4A:upperboundsLkRhats},
we are now interested in finding upper bounds for $\LkAC(\widecheck{\dub})$ and $\LkAC(\widecheck{\dub}_i)$.

\begin{proposition}\label{pr:step4B:upperboundsLkRhats}
Let $\D$ be a dense countably infinite mesh in $\II^n$ and fix some integer $k\in[n]$.
Let $A,C \colon \D \to \II$ be functions with $A \le C$, $L_k^{(A,C)}(\dub)\ge 0$ for all $\dub\in \R_k(\D)$, and $\dkD^{(A,C)}(\dd)=0$ for all $\dd\in\D$.
Furthermore, assume that $A(\vv) = C(\vv)$ for all vertices $\vv$ of the unit cube $\II^n$ and that there exists a $k$-box $\skb\in\R_k(\D)$ with
$\Lk{C,C}(\skb)=v<0$ such that there are vertices $\x_i$ with $m_{\skb}(\x_i)=-1$ and $A(\x_i)<C(\x_i)$ for all $i\in[s]$
and some $s\in[2^{k-1}]$,
while for all other vertices $\x_i \in \ver \skb $ with $m_{\skb}(\x_i)=-1$ we have $A(\x_i) = C(\x_i)$.
\begin{itemize}
\item[\textup{(i)}] For each $i\in[s]$ and each $\widecheck{\dub}_i$ as defined by~\eqref{eq:step4B:DefRi}, the following holds
\begin{equation}\label{eq:step4B:upperboundsLkRhats:Ri}
    \LkAC(\widecheck{\dub}_i)\le m\cdot \big(C(\x_i)-A(\x_i)\big)+m\cdot \Lk{C,C}(\skb)
        +\sum_{l\in[s]\setminus\{i\}} m_l\cdot \LkAC(\dub_l).
\end{equation}
\item[\textup{(ii)}] For $\widecheck{\dub}$ as defined by~\eqref{eq:step4B:DefR0} the following holds
\begin{equation}\label{eq:step4B:upperboundsLkRhats:R0}
    \LkAC(\widecheck{\dub})\le m\cdot\Lk{C,C}(\skb)+\sum_{l=1}^s m_l\cdot \LkAC(\dub_l).
\end{equation}
\end{itemize}
\end{proposition}

\begin{proof}
  When showing the validity of~\eqref{eq:step4B:upperboundsLkRhats:Ri},  the contribution of each $\dd\in\D$ to both sides of the inequality can be considered in analogy to the scenario when constructing $C$ from below, i.e., starting from $A$ (compare also the proof of Proposition~\ref{pr:step4A:upperboundsLkRhats}). The cases to be checked are (1) $\dd=\x_i$, (2) $\dd=\x_j$ with $j\in[s]\setminus\{i\}$ and (3) $\dd\in\D\setminus\{\x_j\mid j\in[s]\}$, i.e., distinguishing whether or not $\dd$ is one of the vertices of $\skb$ with negative multiplicity and different values with respect to $A$ and $C$ or not.
  
  In order to show that~\eqref{eq:step4B:upperboundsLkRhats:R0} holds,  only the contribution of $\dd\in\{\x_1, \dots, \x_s\}$, i.e., for elements $\x_j\in\D$ with $m_{\skb}(\x_j)=-1$ and $A(\x_j)< C(\x_j)$ for all $j\in[s]$
  to both sides of the corresponding inequality needs slightly different arguments compared with the situation when constructing $C$ from $A$ (see also the proof of Proposition~\ref{pr:step4A:upperboundsLkRhats}). We briefly discuss these differences:
  
  \emph{Case} $1$: $\dd\in\{\x_1,\dots, \x_s\}$, i.e., $m_{\skb}(x_j)=-1$ and $A(\x_j)<C(\x_j)$ for all $j\in[s]$.
  If $m_{\widecheck{\dub}}(\x_j)<0$ then also $m_{\widecheck{\dub}_j}(\x_j)=-m+m_{\widecheck{\dub}}(\x_j)<0$ and, as a consequence, we get
\begin{equation*}\POkD^{(A,C)}(\x_j)=\inf_{\substack{\dub\in\R_k(\D)\\ m_\dub(\x_j)<0}}\frac{\Lk{A,C}(\dub)}{|m_\dub(\x_j)|}\le \frac{\Lk{A,C}(\widecheck{\dub}_j)}{|m_{\widecheck{\dub}_j}(\x_j)|}.
\end{equation*}
Since $\PMkDAC(\x_j)=0$, Proposition~\ref{pr:step2} implies
\begin{equation*}
    |m_{\widecheck{\dub}_j}(\x_j)|\cdot(C(\x_j)-A(\x_j))\le |m_{\widecheck{\dub}_j}(\x_j)|\cdot(\POkD^{(A,C)}(\x_j)+\PMkD^{(A,C)}(\x_j))\le \LkAC(\widecheck{\dub}_j).
\end{equation*}
Moreover, by~\eqref{eq:step4B:upperboundsLkRhats:Ri} and in analogy to the proof of Proposition~\ref{pr:step4A:upperboundsLkRhats} we may argue that
\begin{align*}
    |m_{\widecheck{\dub}_j}(\x_j)|\cdot(C(\x_j)-A(\x_j))
        &\le m\cdot (C(\x_j)-A(\x_j))+m\cdot \Lk{C,C}(\skb) \\
        &\quad{} +\sum_{l\in[s]\setminus\{j\}}m_l\cdot \LkAC(\dub_l)\\
        &<m\cdot (C(\x_j)-A(\x_j))+m\cdot v+\sum_{l\in[s]\setminus\{j\}} m_l\cdot |m_{\dub_l}(\x_l)|\cdot \frac{|v|}{s} \\
        &=m\cdot (C(\x_j)-A(\x_j))- \frac{m\cdot |v|}{s}.
\end{align*}
Since $m_{\widecheck{\dub}}(\x_j)=m_{\widecheck{\dub}_j}(\x_j)+m$, we obtain the contradiction
\begin{equation*}
    0\le (-1)\cdot \underbrace{m_{\widecheck{\dub}}(\x_j)}_{<0}\cdot (C(\x_j)-A(\x_j))
=\underbrace{(-m_{\widecheck{\dub}_j}(\x_j)-m)}_{=|m_{\widecheck{\dub}_j}(\x_j)|-m}\cdot (C(\x_j)-A(\x_j))
    < -\frac{m\cdot |v|}{s}<0,
\end{equation*}
showing that $m_{\widecheck{\dub}}(\x_j)\ge 0$ for all 
$l\in[s]$.
Therefore, for the contribution of any $\x_j$ with $j\in[s]$ 
to the left-hand side of~\eqref{eq:step4B:upperboundsLkRhats:R0} we obtain 
\begin{align*}
    m_{\widecheck{\dub}}(\x_j)\cdot C(\x_j)&=\left(m\cdot m_{\skb}(\x_j)+\sum_{l=1}^s m_l\cdot m_{\dub_l}(\x_j)\right)\cdot C(\x_j)\\
        &\le m\cdot m_{\skb}(\x_j)\cdot C(\x_j)+\sum_{l=1}^s m_l \cdot\max\{m_{\dub_l}(\x_j)\cdot A(\x_j),m_{\dub_l}(\x_j)\cdot C(\x_j)\}.
\end{align*}

\emph{Case} $2$:  $\dd\in\D\setminus\{\x_1,\dots, \x_s\}$ with $m_{\skb}(\dd)=-1$ and fulfilling $A(\dd)=C(\dd)$. This case can  be handled in analogy to the situation when constructing $C$ from $A$ (see Proposition~\ref{pr:step4B:upperboundsLkRhats}). 
\end{proof}

Based on these results it can be shown that the function $C$ obtained as pointwise limit of the sequence $(B^{(i)})_{i \in\NN}$ does not only have all the properties mentioned in Proposition~\ref{pr:step3B}, but is also $k$-increasing on $\D$.

\begin{proposition}\label{pr:step4B}
Let $\D$ be a dense countably infinite mesh in $\II^n$ and fix some integer $k\in[n]$.  
Let $A, B \colon \D \to \II$ be functions with $A \le B$ and $L_k^{(A,B)}(\dub)\ge 0$ for all $\dub\in \R_k(\D)$. Furthermore, assume that $A(\vv) = B(\vv)$ for all vertices $\vv$ of the unit cube $\II^n$.
Let $C\colon\D\to\II$ be the pointwise limit of the sequence $(B^{(i)})_{i\in\NN}$ as given by~\eqref{eq:step3B:DefBis}.
Then $C$ is $k$-increasing on $\D$, i.e., $\Lk{C,C}(\dub)\ge 0 \text{ for all } \dub\in\R_k(\D)$.
\end{proposition}

The proof of Proposition~\ref{pr:step4B} can be carried out in analogy to the proof of Proposition~\ref{pr:step4A}, using some of the results of Lemma~\ref{le:step4B:existenceRi} and Propositions~\ref{pr:step3B} and~\ref{pr:step4B:upperboundsLkRhats}. 

\section{From a dense mesh to the unit cube} \label{se:extension}

Given functions $A\leq B$ defined on $\II^n$, in both Propositions~\ref{pr:step3A} and~\ref{pr:step3B} a function $C$, defined on $\D$, is constructed satisfying 
$A(\x)\leq C(\x)\leq B(\x)$
for all $\x\in\D$. Propositions~\ref{pr:step4A} and~\ref{pr:step4B} show that  the function $C$ obtained in this way is $k$-increasing on $\D$ for some $k\in[n]$. 
Thus we can extend $C$ to the entire unit cube $\II^n$ and show that this extension is still $k$-increasing and lies between $A$ and $B$ (on the whole~$\II^n$).

\begin{proposition} \label{pr:step5:generalfunctions}
    Let $A,B \colon \II^n \to \II$ be standardized functions with $A \leq B$ and fix some integer $k\in[n]$. Suppose that at least one of the functions $A$ and $B$ satisfies Condition~{\bf S} with set $S$.
    Furthermore, let $\D \subseteq \II^n$ be a dense countably infinite mesh with $S^n \subseteq \D$ and $C \colon \D \to \II$ a $k$-increasing function such that $A(\dd) \leq C(\dd) \leq B(\dd)$ for all $\dd \in \D$.
    Then $C$ can be extended to a $k$-increasing function $\widehat{C} \colon \II^n \to \II$ such that $A(\x) \leq \widehat{C}(\x) \leq B(\x)$ for all $\x \in \II^n$.
\end{proposition}

\begin{proof}
    Suppose first that the function $A$ satisfies Condition~{\bf S},
    let $\D = \delta_1 \times \delta_2 \times \dots \times \delta_n$ and define the function $\widehat{C} \colon \II^n \to \II$ by
    \begin{equation}\label{eq:Chat}
     \widehat{C}(\x)=\sup\{C(\dd) \mid \dd \in \cint{\mathbf{0},\x}\cap \D\}.   
    \end{equation}
    Note that the set on the right-hand side is non-empty because of $\mathbf{0} \in \D$. Since $C$ is grounded and $k$-increasing on $\D$, it is $1$-increasing on $\D$. Hence, for each $\dd \in \D$ we have $\widehat{C}(\dd) = C(\dd)$, and $\widehat{C}$ is an extension of $C$ to $\II^n$.

    We claim that $\widehat{C}$ is $k$-increasing on $\II^n$.
    Let $R=\cint{\x,\y} \subseteq \II^n$ be a $k$-box.
    Denote the vertices of $R$ by $\vv_1,\vv_2,\dots,\vv_r$, where $r=2^k$.
    We may assume that $m_R(\vv_j)=1$ if $j\in[\frac r2]$  and $m_R(\vv_j)=-1$ if 
    $j\in[r]\setminus [\frac r2]$.
    Let $\varepsilon >0$.
    For every $j \in [r]$ there exists $\dd_j \in \cint{\mathbf{0},\vv_j} \cap \D$ such that $\widehat{C}(\vv_j)-\frac{2\varepsilon}{r} < C(\dd_j)$.
    Using these points $\dd_j$ we construct a $k$-box $\widehat{R}$ with vertices in $\D$ which approximates $R$.
    For each $i \in [n]$ let $J_i^1=\{j \in [r] \mid (\vv_j)_i=x_i\}$ and $J_i^2=\{j \in [r] \mid (\vv_j)_i=y_i\}$, so that $J_i^1 \cup J_i^2 =[r]$.
    If $x_i=y_i$, we define $\widehat{x}_i=\max\{(\dd_j)_i \mid j \in [r]\} \in \delta_i$ and $\widehat{y}_i=\widehat{x}_i$,
    so that $\widehat{x}_i=\widehat{y}_i \leq x_i=y_i$.
    If $x_i<y_i$, we define $\widehat{x}_i=\max\{(\dd_j)_i \mid j \in J_i^1\} \in \delta_i$,
    choose $d_i \in \opint{x_i,y_i} \cap \delta_i$,
    and define $\widehat{y}_i=\max\{d_i,\max\{(\dd_j)_i \mid j \in J_i^2\}\} \in \delta_i$,
    so that $\widehat{x}_i\leq x_i<\widehat{y}_i\leq y_i$.
    Finally, we put $\widehat{\x}=(\widehat{x}_1,\widehat{x}_2,\dots,\widehat{x}_n)$, $\widehat{\y}=(\widehat{y}_1,\widehat{y}_2,\dots,\widehat{y}_n)$, and $\widehat{R}=\cint{\widehat{\x},\widehat{\y}}$, the latter being a $k$-box with vertices in $\D$.
    Denote the vertices of $\widehat{R}$ by $\widehat{\dd}_1,\widehat{\dd}_2,\dots,\widehat{\dd}_r$ in such a way that for all $i \in [n]$ and $j \in [r]$ we have $\bigl(\widehat{\dd}_j\bigr)_i=\widehat{x}_i$ if and only if $(\vv_j)_i=x_i$, i.e., if and only if $j \in J_i^1$.
    Then $m_{\widehat{R}}\bigl(\widehat{\dd}_j\bigr)=m_R(\vv_j)$ and $\dd_j \leq \widehat{\dd}_j \leq \vv_j$ for all $j \in [r]$. By \eqref{eq:Chat}, this implies $\widehat{C}(\vv_j) \geq C(\widehat{\dd}_j)$.
    Therefore we have the following inequality,
    \begin{align*}
        V_{\widehat{C}}(R) &=\sum_{j=1}^{r/2} \widehat{C}(\vv_j)-\sum_{j=r/2+1}^r \widehat{C}(\vv_j)\\
        &>
        \sum_{j=1}^{r/2} C\left(\widehat{\dd}_j\right)-\sum_{j=r/2+1}^r \left(C(\dd_j)+\frac{2\varepsilon}{r}\right)
        =\sum_{j=1}^{r/2} C\left(\widehat{\dd}_j\right)-\sum_{j=r/2+1}^r C(\dd_j)-\varepsilon,
    \end{align*}
    and since $C$ is $1$-increasing and $k$-increasing on $\D$ we obtain
    \begin{equation*}
    V_{\widehat{C}}(R) >\sum_{j=1}^{r/2} C\left(\widehat{\dd}_j\right)-\sum_{j=r/2+1}^r C\left(\widehat{\dd}_j\right)-\varepsilon =V_C\left(\widehat{R}\right)-\varepsilon \geq -\varepsilon.
    \end{equation*}
    Sending $\varepsilon$ to $0$, we get $V_{\widehat{C}}(R) \geq 0$ for all $k$-boxes $R \subseteq \II^n$, i.e., $\widehat{C}$ is $k$-increasing on $\II^n$.

    We have $A(\dd) \leq \widehat{C}(\dd) \leq B(\dd)$ for any $\dd \in \D$ by assumption. For any $\x \in \II^n$ and $\dd \in \cint{\mathbf{0},\x} \cap \D$ we have $C(\dd) \leq B(\dd) \leq B(\x)$, so $\widehat{C}(\x) \leq B(\x)$ by \eqref{eq:Chat}. To complete the proof it remains to show that $A(\x) \leq \widehat{C}(\x)$ for each $\x \in \II^n \setminus \D$.

    Let $\x \in \II^n \setminus \D$. Since $\x \notin \D$, at least one coordinate $x_i$ of $\x$ does not belong to $\delta_i$. We may assume without loss of generality that, for some $m \in [n]$, we have $x_j \notin \delta_j$ for $j \in [m]$ and $x_j \in \delta_j$ for $j\in[n]\setminus [m]$.
    For any $\dd \in \cint{\mathbf{0},\x} \cap \D$ let $\widehat{\dd} = (d_1, \dots, d_m, x_{m+1}, \dots, x_n) \in \cint{\mathbf{0},\x} \cap \D$.
    The function $f_1 \colon t_1 \longmapsto A(t_1,d_2, \dots, d_m, x_{m+1}, \dots, x_n)$ is continuous at $t_1 = x_1$ by Condition~{\bf S}, since $x_1 \notin \delta_1$ and $S\subseteq \delta_1$. Since $f_1$ is also increasing and $\delta_1$ is dense in $\II$, it follows that
    \begin{equation*} 
    A(x_1,d_2, \dots, d_m, x_{m+1}, \dots, x_n) = \sup\bigl\{A(t_1,d_2, \dots, d_m, x_{m+1}, \dots, x_n) \bigm| t_1 \in \cint{0, x_1} \cap \delta_1\bigr\}.
    \end{equation*}
    The function $f_2 \colon t_2 \longmapsto A(x_1,t_2,d_3, \dots, d_m, x_{m+1}, \dots, x_n)$ is continuous at $t_2 = x_2$, increasing, and $\delta_2$ is dense in $\II$, so
    \begin{align*}
       A(x_1, x_2, d_3, &\dots, d_m, x_{m+1}, \dots, x_n) \\[1ex]
       &= \sup\bigl\{A(x_1, t_2, d_3, \dots, d_m, x_{m+1}, \dots, x_n) \bigm| t_2 \in \cint{0, x_2} \cap \delta_2\bigr\} \\[1ex]
       &= \sup\bigl\{A(t_1, t_2, d_3, \dots, d_m, x_{m+1}, \dots, x_n) \bigm| t_1 \in \cint{0, x_1} \cap \delta_1, t_2 \in \cint{0, x_2} \cap \delta_2\bigr\} .
    \end{align*}
    Continuing inductively up to index $m$, and at the last step using the increasing function $f_m \colon t_m \longmapsto A(x_1,\dots, x_{m-1}, t_m, x_{m+1}, \dots, x_n)$ which is continuous at $t_m=x_m$, we obtain
    \begin{align*}
       A(\x) &= \sup\bigl\{A(x_1, \dots, x_{m-1}, t_m, x_{m+1}, \dots, x_n) \bigm| t_m \in \cint{0, x_m} \cap \delta_m\bigr\} \\[1ex]
       &= \sup\bigl\{A(t_1, \dots, t_m, x_{m+1}, \dots, x_n) \bigm| t_j \in \cint{0, x_j} \cap \delta_j \text{ for all } j \in [m]\bigr\} \\[1ex]
       &= \sup\Bigl\{A(\widehat{\dd}) \Bigm| \dd \in \cint{\mathbf{0},\x} \cap \D\Bigr\} 
       = \sup\bigl\{A(\dd) \bigm| \dd \in \cint{\mathbf{0},\x} \cap \D\bigr\} \\[1ex]
       &\leq \sup\bigl\{C(\dd) \bigm| \dd \in \cint{\mathbf{0},\x} \cap \D\bigr\} 
       = \widehat{C}(\x), 
    \end{align*}
    completing the proof when $A$ satisfies Condition~{\bf S}.
    Suppose now that the function $B$ satisfies Condition~{\bf S}. In this case we define for every $\x \in \II^n$
    \begin{equation*} 
    \widetilde{C}(\x)=\inf\{C(\dd) \mid \dd \in \cint{\x,\1} \cap \D\}.  
    \end{equation*} 
    The function $\widetilde{C}$ is another extension of $C$ to $\II^n$. Similarly as in the previous case we show that also~$\widetilde{C}$ is $k$-increasing. It is obvious that $A(\x) \leq \widetilde{C}(\x)$ for any $\x \in \II^n$. In order to prove $\widetilde{C}(\x) \leq B(\x)$ we use Condition~{\bf S} for the function $B$ to show that    \begin{equation*} B(\x)=\inf\{B(\dd) \mid \dd \in \cint{\x,\1} \cap \D\} \end{equation*} 
    in a similar way as above. 
\end{proof}

When formulating the counterpart of Proposition~\ref{pr:step5:generalfunctions} for the case of semicopulas, we can drop Condition~\textbf{S}.

\begin{proposition} \label{pr:step5:semicopulas}
Consider two $n$-variate semicopulas $A,B\colon \II^n\to\II$ with $A\le B$ and fix some integer $k\in[n]\setminus\{1\}$. 
Further, let $\D \subseteq \II^n$ be a dense countably infinite mesh and $C \colon \D \to \II$ a $k$-increasing function such that $A(\dd) \leq C(\dd) \leq B(\dd)$ for all $\dd \in \D$.
    Then $C$ can be extended to a $k$-increasing function $C' \colon \II^n \to \II$ such that $A(\x) \leq C'(\x) \leq B(\x)$ for all $\x \in \II^n$.
\end{proposition}

\begin{proof}
The functions $A$ and $B$ are semicopulas and $C$ lies between them, so $C$ is grounded and has uniform marginals on $\D$. 
  Since $C$ is grounded and $k$-increasing on $\D$ for some $k\geq2$, it is also 1-Lipschitz.  A 1-Lipschitz function has a unique continuous extension to the closure of its definition set. The fact that $\D$ is dense in $\II^n$ makes the unique extension $C'$ defined for all $\x\in \II^n$. Note that $C'$ coincides with the extension $\widehat{C}$ of $C$ defined in \eqref{eq:Chat}.
  However, when proving $k$-increasingness of $\widehat{C}$ in the proof of Proposition~\ref{pr:step5:generalfunctions}, Condition~{\bf S} was not utilized. Thus the same proof can be used here to show that $C'$ is $k$-increasing.
  
  Furthermore, since $C$ is 1-Lipschitz on $\D$, the extension $C'$ also coincides with the extension $\widetilde{C}$ from the proof of Proposition~\ref{pr:step5:generalfunctions}. The proofs that $A(\x)\le \widetilde{C}(\x)$ and $\widehat{C}(\x)\le B(\x)$ for all $\x \in \II^n$ did not require Condition~{\bf S}, hence $A(\x) 
  \le C'(\x)\le B(\x)$.
\end{proof}

\section{Proofs of the main theorems}\label{se:proofofmaintheorem}

Now we can combine our findings to prove our main results, i.e., Theorems~\ref{th:maintheorem-2} and~\ref{th:maintheorem} which were stated in Section~\ref{se:maintheorem}.

\begin{proof}[Proof of Theorem~\textup{\ref{th:maintheorem-2}}]
Let $A,B \colon \II^n \to \II$ be standardized functions with $A \leq B$. Suppose that at least one of the functions $A$ and $B$ satisfies Condition~{\bf S} with a set $S$.

We first show that \textup{(i)} implies \textup{(ii)}. Let $C\colon \II^n \to \II$ be a $k$-increasing function with $A \le C \le B$ and $\dub \in \R_k(\II^n)$.
Then
\begin{align*}
\Lk{A,B}(\dub) &=\sum_{\substack{\y\in\II^n\\m_\dub(\y)>0}} m_\dub(\y)B(\y) + \sum_{\substack{\y\in\II^n\\ m_\dub(\y)<0}} m_\dub(\y)A(\y)\\
&\ge \sum_{\substack{\y\in\II^n\\m_\dub(\y)>0}} m_\dub(\y)C(\y) + \sum_{\substack{\y\in\II^n\\ m_\dub(\y)<0}} m_\dub(\y)C(\y) = L_k^{(C,C)}(\dub) \ge 0.
\end{align*}
To prove that \textup{(ii)} implies \textup{(i)}, suppose $L_k^{(A,B)}(\dub)\ge 0$ for all $\dub\in \R_k(\II^n)$.
Let $\D$ be a dense countably infinite mesh that contains $S^n$. Then $L_k^{(A,B)}(\dub)\ge 0$ for all $\dub\in \R_k(\D)$. By Proposition~\ref{pr:step3A} there exists a function $C \colon \D \to \II$ defined by \eqref{eq:step3A:DefAis} such that $A(\dd) \le C(\dd) \le B(\dd)$ for all $\dd \in \D$.
Since $A$ and $B$ are standardized functions, we have $A(\vv) = B(\vv)$ for all vertices $\vv$ of the unit cube $\II^n$.
Proposition~\ref{pr:step4A} implies that $C$ is $k$-increasing on $\D$.
Hence, by Proposition~\ref{pr:step5:generalfunctions}, function $C$ can be extended to a $k$-increasing function $C \colon \II^n \to \II$ such that $A(\x) \le C(\x) \le B(\x)$ for all $\x \in \II^n$.
\end{proof}

In the proof of Proposition~\ref{pr:step3A} we arranged the elements of $\D$ into a sequence $(\dd_i)_{i \in \NN}$. The order is not important for the proof to work.
Note, however, that the obtained function $C$, constructed from below, satisfies
\begin{equation}\label{eq:Cd1-below}
\begin{split}
C(\dd_1) &=A'(\dd_1)=A(\dd_1)+\gkD^{(A,B)}(\dd_1)=A(\dd_1)+\min\{B(\dd_1)-A(\dd_1),\POkD^{(A,B)}(\dd_1)\}\\[1ex]
&=\min\{B(\dd_1),A(\dd_1)+\POkD^{(A,B)}(\dd_1)\},
\end{split}
\end{equation}
where the function $A'$ is defined as in Proposition~\ref{pr:step1A}.

The proof of Theorem~\ref{th:maintheorem-2} could analogously be done with the use of Propositions~\ref{pr:step3B} and \ref{pr:step4B}, in which case the obtained function $C$, constructed from above, would satisfy
\begin{equation}\label{eq:Cd1-above}
C(\dd_1) =B'(\dd_1)=B(\dd_1)-\dkD^{(A,B)}(\dd_1)=\max\{A(\dd_1),B(\dd_1)-\PMkD^{(A,B)}(\dd_1)\}.
\end{equation}

\begin{proof}[Proof of Theorem~\textup{\ref{th:maintheorem}}]
The proof for $k \ge 2$ is analogous to the proof of Theorem~\ref{th:maintheorem-2} (note that the assumption $k \ge 2$ appears in Proposition~\ref{pr:step5:semicopulas}). The only differences are that we can choose the dense countably infinite mesh $\D$ arbitrarily since we do not have Condition~\textbf{S}, and that we use Proposition~\ref{pr:step5:semicopulas} instead of Proposition~\ref{pr:step5:generalfunctions}.
Note that the obtained function $C$ is a semicopula because it has uniform marginals, a property that it inherits from $A$ and $B$, since
$A \le C \le B$.
If $k=1$ then the proof that \textup{(i)} implies \textup{(ii)} is the same while the opposite direction is trivial since we can take $C=A$.
\end{proof}

\section{Coherence results} \label{se:cogerence}

In the previous sections we were dealing with results related to the avoidance of sure loss. Now we present four consequences concerning coherence (see Definition~\ref{def:avoidance}).
Theorems~\ref{th:coherence-upper-general} and \ref{th:coherence-lower-general} consider the case when the bounds are standardized functions, while Theorems~\ref{th:coherence-upper-semicopulas} and \ref{th:coherence-lower-semicopulas} deal with the case when the bounds are semicopulas. The $k$-coherence for the upper bound is given in Theorems~\ref{th:coherence-upper-general} and \ref{th:coherence-upper-semicopulas}, while the $k$-coherence for the lower bound is considered in Theorems~\ref{th:coherence-lower-general} and \ref{th:coherence-lower-semicopulas}.

\begin{theorem}\label{th:coherence-upper-general}
Let $A,B \colon \II^n \to \II$ be standardized functions with $A \leq B$  and fix some integer $k\in[n]$. 
Suppose that at least one of the functions $A$ and $B$ satisfies Condition~{\bf S} and that  
    $L_k^{(A,B)}(\dub)\ge 0$ for all $\dub\in\R_k(\II^n)$. 
Then the following are equivalent:
\begin{itemize}
    \item[\textup{(i)}] for all $\x\in\II^n: \quad \POk^{(A,B)}(\x)\ge B(\x)-A(\x)$;
    \item[\textup{(ii)}] for all $\x\in\II^n:$
     \begin{equation*}
         B(\x)=\sup\{C(\x)\mid C\colon\II^n\to\II,\ A\le C\le B,\ C\textrm{ is a $k$-increasing standardized function}\}.
     \end{equation*}
\end{itemize}
\end{theorem}

\begin{proof} (i) $\Rightarrow$ (ii):
Assume that condition~\textup{(i)} holds and fix some $\x\in\II^n$. Choose a dense countable mesh $\D\subset\II^n$ such that $\x\in \D$ and $S^n \subseteq \D$. Note that for all $\dd\in\D$ we have
\begin{equation*}
\POk^{(A,B)}(\dd)=\inf_{\substack{\dub\in\R_k(\II^n),\\ m_\dub(\dd)<0}} \frac{L_k^{(A,B)}(\dub)}{|m_\dub(\dd)|}\le \inf_{\substack{\dub\in\R_k(\D),\\ m_\dub(\dd)<0}} \frac{L_k^{(A,B)}(\dub)}{|m_\dub(\dd)|}=\POkD^{(A,B)}(\dd).
\end{equation*}
We repeat the proof from Theorem~\ref{th:maintheorem-2} by choosing $\dd_1=\x$ in the proof of Proposition~\ref{pr:step3A}. By Equation~\eqref{eq:Cd1-below} we get $C(\x)=B(\x)$ since $\POkD^{(A,B)}(\x) \ge \POk^{(A,B)}(\x)$ by the above and $\POk^{(A,B)}(\x)\ge B(\x)-A(\x)$ by assumption. Doing this for all $\x\in\II^n$ gives us condition~\textup{(ii)}, because any $k$-increasing function between the standardized functions $A$ and $B$ is automatically standardized.

(ii) $\Rightarrow$ (i): Now assume that condition~\textup{(ii)} holds. Fix some $\x\in\II^n$ and some $\varepsilon>0$. By condition~\textup{(ii)} there is a $k$-increasing standardized function $C\colon\II^n\to\II$ such that $A \le C \le B$ and $B(\x)-\varepsilon<C(\x)$. Then
$\POk^{(A,B)}(\x)\ge \POk^{(A,C)}(\x)$ 
and, for any $\dub\in\R_k(\II^n)$ with $m_\dub(\x)<0$, we have
\begin{align*}
L_k^{(A,C)}(\dub)&=\underbrace{L_k^{(C,C)}(\dub)}_{\ge 0}+\sum_{\substack{\y\in\II^n\\m_\dub(\y)<0}}\underbrace{m_\dub(\y)(A(\y)-C(\y))}_{\ge 0}
\ge m_\dub(\x)(A(\x)-C(\x)).
\end{align*}
Hence, we get
\begin{equation*}
    \POk^{(A,C)}(\x)=\inf_{\substack{\dub\in\R_k(\II^n)\\m_\dub(\x)<0}}\frac{L_k^{(A,C)}(\dub)}{|m_\dub(\x)|}
    \ge -(A(\x)-C(\x))
    > B(\x)-A(\x)-\varepsilon
\end{equation*}
and, therefore,
$\POk^{(A,B)}(\x)\ge B(\x)-A(\x)$.    
\end{proof}

When the bounds are semicopulas, Condition \textbf{S} can be omitted and the equivalent condition (\textup{i}) can be given as an equality.

\begin{theorem}\label{th:coherence-upper-semicopulas}
Let $A,B\colon\II^n\to\II$ be semicopulas with $A\le B$ and fix some integer $k\in[n]\setminus\{1\}$. 
Assume that  
    $L_k^{(A,B)}(\dub)\ge 0$ for all $\dub\in\R_k(\II^n)$. 
Then the following are equivalent:
\begin{itemize}
    \item[\textup{(i)}] for all $\x\in\II^n: \quad \POk^{(A,B)}(\x) = B(\x)-A(\x)$;
    \item[\textup{(ii)}] for all $\x\in\II^n:$
     \begin{equation*}
         B(\x)=\sup\{C(\x)\mid C\colon\II^n\to\II,\ A\le C\le B,\ C\textrm{ is a $k$-increasing semicopula}\};
     \end{equation*}
\end{itemize}
Each of the conditions \textup{(i)} and \textup{(ii)} implies that $B$ is a quasi-copula.
\end{theorem}

\begin{proof}
Assume that the condition~\textup{(i)} holds.
It trivially follows that
\begin{equation} \label{eq:8.1}
 \textrm{for all} \ \x\in\II^n: \quad \POk^{(A,B)}(\x) \ge B(\x)-A(\x).
\end{equation}
We prove the condition~\textup{(ii)} in the same way as in Theorem~\ref{th:coherence-upper-general}, but with the use of Theorem~\ref{th:maintheorem} instead of Theorem~\ref{th:maintheorem-2}.

Now, assume that condition~\textup{(ii)} holds. 
We verify condition~\eqref{eq:8.1}, again, in the same way as in Theorem~\ref{th:coherence-upper-general}.
Every $k$-increasing semicopula $C$ is also $2$-increasing and thus 1-Lipschitz. By \textup{(ii)} $B$ is a supremum of 1-Lipschitz functions, so it is 1-Lipschitz and thus continuous. By Lemma~\ref{lem:cont} it holds that $\POk^{(A,B)}(\x)\le B(\x)-A(\x)$. Together with \eqref{eq:8.1}, condition~\textup{(i)}  follows.
\end{proof}

The proofs of the following two theorems rely on the proofs of Theorems~\ref{th:maintheorem-2} and \ref{th:maintheorem} by constructing~$C$ from above, replacing the function $\gkD^{(A,B)}$ by $\dkD^{(A,B)}$ and using Equation~\eqref{eq:Cd1-above} instead of Equation~\eqref{eq:Cd1-below} (see also Section~\ref{se:CbyloweringB_short}).

\begin{theorem}\label{th:coherence-lower-general}
Let $A,B \colon \II^n \to \II$ be standardized functions with $A \leq B$  and fix some integer $k\in[n]$. 
Suppose that at least one of the functions $A$ and $B$ satisfies Condition~{\bf S} and that  
    $L_k^{(A,B)}(\dub)\ge 0$ for all $\dub\in\R_k(\II^n)$. 
Then the following are equivalent:
\begin{itemize}
    \item[\textup{(i)}] for all $\x\in\II^n: \quad \PMk^{(A,B)}(\x)\ge B(\x)-A(\x)$;
    \item[\textup{(ii)}] for all $\x\in\II^n\colon$
     \begin{equation*}
         A(\x)=\inf\{C(\x)\mid C\colon\II^n\to\II,\ A\le C\le B,\ C\textrm{ is a $k$-increasing standardized function}\}.
     \end{equation*}
\end{itemize}
\end{theorem}

\begin{theorem}\label{th:coherence-lower-semicopulas}
Let $A,B\colon\II^n\to\II$ be semicopulas with $A\le B$ and fix some integer $k\in[n]\setminus\{1\}$. 
Assume that  
    $L_k^{(A,B)}(\dub)\ge 0$ for all $\dub\in\R_k(\II^n)$. 
Then the following are equivalent:
\begin{itemize}
    \item[\textup{(i)}] for all $\x\in\II^n: \quad \PMk^{(A,B)}(\x)= B(\x)-A(\x)$;
    \item[\textup{(ii)}] for all $\x\in\II^n:$
     \begin{equation*}
         A(\x)=\inf\{C(\x)\mid C\colon\II^n\to\II,\ A\le C\le B,\ C\textrm{ is a $k$-increasing semicopula}\};
     \end{equation*}
\end{itemize}
Each of the conditions \textup{(i)} and \textup{(ii)} implies that $A$ is a quasi-copula.
\end{theorem}

\section*{Concluding remarks}

We discuss $k$-increasing $n$-variate standardized functions, semicopulas and quasi-copulas. The case of $k=n$ refers to the characterization of $n$-variate copulas and the problem of relating a true copula to an imprecise copula has been solved in~\cite{OmlSto20a}. In~\cite{AriDeB19,AriMesDeB17,AriMesDeB20,MueSca00} several aspects of $k$-increasing $n$-quasi-copulas have been discussed, in particular for $k=2$ covering the case of supermodularity. One of our aims has been to further reduce the conditions imposed on the functions we start with. Thus we work with a larger class of functions but still obtain results on their avoidance of sure loss, i.e., provide a characterization of the existence of a $k$-increasing $n$-variate standardized function and of a semicopula. Furthermore, we also show coherence results in this general setting. 
We expect that our results contribute to a deeper understanding of probability and imprecise probabilities. 

\section*{Acknowledgments}

The support by the WTZ AT-SLO grant SI 12/2020 of the OeAD (Austrian Agency for International Cooperation in Education and Research) and grant  BI-AT/20-21-009 of the ARIS (Slovenian Research and Innovation Agency) is gratefully acknowledged. 
Damjana Kokol Bukovšek, Blaž Mojškerc, and Nik Stopar acknowledge financial support from the ARIS (Slovenian Research and Innovation Agency, research core funding No. P1-0222).


\begin{thebibliography}{10}

\bibitem{AlsNelSch93}
C.~Alsina, R.~B. Nelsen, B.~Schweizer, On the characterization of a class of binary operations on
  distribution functions, Statist. Probab. Lett. 17 (1993) 85--89. \url{https://doi.org/10.1016/0167-7152(93)90001-Y}.

\bibitem{AriDeB19}
Arias-Garc\'{\i}a JJ, {De Baets} B (2019) On the lattice structure of
  the set of supermodular quasi-copulas. Fuzzy Sets and Systems 354:74--83. \url{https://doi.org/10.1016/j.fss.2018.03.013}.

\bibitem{AriMesDeB17}
{Arias-Garc\'{\i}a} JJ, Mesiar R, {De Baets} B (2017) The unwalked path
  between quasi-copulas and copulas: {S}tepping stones in higher dimensions.
  Int. J. Approx. Reason. 80:89--99.
  \url{https://doi.org/10.1016/j.ijar.2016.08.009}.

\bibitem{AriMesDeB20}
{Arias-Garc\'{\i}a} JJ, Mesiar R, {De Baets} B (2020) A hitchhiker's
  guide to quasi-copulas. Fuzzy Sets and Systems 393:1--28.
  \url{https://doi.org/10.1016/j.fss.2019.06.009}.

\bibitem{ArtDelEbeHea99}
Artzner P, Delbaen F,  Eber JM, Heath, D (1999) Coherent measures of
  risk. Math. Finance 9:203--228.
  \url{https://doi.org/10.1111/1467-9965.00068}.

\bibitem{AugCooDeCTro14}
Augustin T, Coolen FPA, {de Cooman} G,  Troffaes MCM eds (2014)
  Introduction to Imprecise Probabilities. (John Wiley \& Sons, Chichester). 
  \url{https://doi.org/10.1002/9781118763117}.

\bibitem{BasSpi05}
Bassan B, Spizzichino F (2005) Relations among univariate aging, bivariate
  aging and dependence for exchangeable lifetimes. J. Multivariate Anal.
  93:313--339. \url{https://doi.org/10.1016/j.jmva.2004.04.002}.

\bibitem{Cho54}
Choquet G (1954) Theory of capacities. Ann. Inst. Fourier (Grenoble) 5:131--295. \url{http://www.numdam.org/item?id=AIF\_1954\_\_5\_\_131\_0}.

\bibitem{Den94}
Denneberg D (1994) Non-additive Measure and Integral  (Kluwer, Dordrecht). \url{https://doi.org/10.1007/978-94-017-2434-0}.

\bibitem{DurFerTru16}
Durante F, Fern\'{a}ndez-S{\'a}nchez J, Trutschnig W (2016)
Baire category results for quasi-copulas.
Depend. Model. 4:215--223. \url{https://doi.org/10.1515/demo-2016-0012}.

\bibitem{DurKleMesSem07}
Durante F, Klement EP, Mesiar R, Sempi C (2007) Conjunctors and their
  residual implicators: characterizations and construction methods. Mediterr.
  J. Math. 4:343--356. \url{https://doi.org/10.1007/s00009-007-0122-1}.

\bibitem{DurMesPap08}
Durante F, Mesiar R, Papini PL (2008) The lattice-theoretic structure
  of the sets of triangular norms and semi-copulas. Nonlinear Anal. 69:46--52. \url{https://doi.org/10.1016/j.na.2007.04.039}.

\bibitem{DurQueSem06}
Durante F, Quesada-Molina JJ, Sempi C (2006) Semicopulas:
  characterizations and applicability. Kybernetika (Prague) 42:287--302. \url{https://www.kybernetika.cz/content/2006/3/287}.

\bibitem{DurSem05}
Durante F, Sempi C (2005) Semicopul{\ae}. Kybernetika (Prague) 41:315--328. \url{https://www.kybernetika.cz/content/2005/3/315}.

\bibitem{GenQueRodSem99}
C.~Genest, J.~J. Quesada-Molina, J.~A. Rodr{\'{\i}}guez-Lallena, C.~Sempi, A characterization of quasi-copulas, J.~Multivariate Anal. 69 (1999) 193--205. \url{https://doi.org/10.1006/jmva.1998.1809}.

\bibitem{GilSch94}
Gilboa I, Schmeidler D (1994) Additive representations of non-additive
  measures and the {C}hoquet integral. Ann. Oper. Res. 52:43--65. \url{https://doi.org/10.1007/BF02032160}.

\bibitem{GilSch95}
Gilboa I, Schmeidler D (1995) Canonical representation of set functions. Math. Oper. Res. 20:197--212. \url{https://doi.org/10.1287/moor.20.1.197}.

\bibitem{KleKolMesSam17}
Klement EP, Koles\'{a}rov\'{a} A, Mesiar R,  Saminger-Platz S (2017) On the role of ultramodularity and {S}chur concavity in the
  construction of binary copulas. J. Math. Inequal. 11:361--381. \url{https://doi.org/10.7153/jmi-11-32}.

\bibitem{KleMesPap10}
Klement EP, Mesiar R, Pap E (2010) A universal integral as common
  frame for {C}hoquet and {S}ugeno integral. IEEE Trans. Fuzzy Syst. 18:178--187. \url{https://doi.org/10.1109/TFUZZ.2009.2039367}.

\bibitem{MarMon05}
Marinacci M, Montrucchio L (2005) Ultramodular functions.
  Math. Oper. Res. 30:311--332. 
  \url{https://doi.org/10.1287/moor.1040.0143}.

\bibitem{MarMon08}
Marinacci M, Montrucchio L (2008) On concavity and supermodularity. J.
  Math. Anal. Appl. 344:642--654.
  \url{https://doi.org/10.1016/j.jmaa.2008.03.009}.

\bibitem{MarOlk79}
Marshall AW, Olkin I  (1979) Inequalities: {T}heory of {M}ajorization and
  {I}ts {A}pplications. (Academic Press [Harcourt Brace Jovanovich,
  Publishers], New York-London). \url{https://doi.org/10.1007/978-0-387-68276-1}.

\bibitem{MonMirDes20}
Montes I, Miranda E, Destercke S (2020) Unifying neighbourhood and distortion models: part II – new models and synthesis. Int. J. General Syst. 49:636--674.
\url{https://doi.org/10.1080/03081079.2020.1778683}.


\bibitem{MonMirMon14a}
Montes I, Miranda E, Montes S (2014) Decision making with imprecise
  probabilities and utilities by means of statistical preference and stochastic
  dominance. European J. Oper. Res. 234:209--220.
  \url{https://doi.org/10.1016/j.ejor.2013.09.013}.

\bibitem{MonMirPelVic15}
Montes I, Miranda E, Pelessoni R, and Vicig P (2015)
Sklar's theorem in an imprecise setting.
Fuzzy Sets and Systems 278:48--66, 2015. \url{https://doi.org/10.1016/j.fss.2014.10.007}.

\bibitem{MueSca00}
M\"uller A, Scarsini M (2000) Some remarks on the supermodular order. J. Multivariate Anal. 73:107--119.
  \url{https://doi.org/10.1006/jmva.1999.1867}.

\bibitem{MurSug91}
Murofushi T, Sugeno M (1991) Fuzzy t-conorm integrals with respect to
  fuzzy measures: generalization of {S}ugeno integral and {C}hoquet integral.
  Fuzzy Sets and Systems 42:57--71.
  \url{https://doi.org/10.1016/0165-0114(91)90089-9}.

\bibitem{OmlSto20a}
Omladi\v{c} M, Stopar N (2020) Final solution to the problem of relating a
  true copula to an imprecise copula. Fuzzy Sets and Systems 393:96--112. \url{https://doi.org/10.1016/j.fss.2019.07.002}.

\bibitem{OmlSto20b}
Omladi\v{c} M, Stopar N (2020) A full scale {S}klar's theorem in the imprecise setting. Fuzzy
  Sets and Systems 393:113--125.
  \url{https://doi.org/10.1016/j.fss.2020.02.001}.

\bibitem{OmlSto22a}
Omladi\v{c} M, Stopar N (2022) Dedekind--{M}ac{N}eille completion of multivariate copulas via
  {ALGEN} method. Fuzzy Sets and Systems 441:321--334.
  \url{https://doi.org/10.1016/j.fss.2021.10.011}.

\bibitem{OmlSto22}
Omladi\v{c} M, Stopar N (2022) Multivariate imprecise {S}klar type theorems. Fuzzy Sets and
  Systems 428:80--101. \url{https://doi.org/10.1016/j.fss.2020.12.002}.

\bibitem{PelVicMonMir13}
Pelessoni R, Vicig P, Montes I, Miranda E (2013) Imprecise copulas and
  bivariate stochastic orders. Proceedings of EUROFUSE 2013 (Oviedo),
  217--224. \url{http://hdl.handle.net/10651/40211}.

\bibitem{PelVicMonMir16}
Pelessoni R, Vicig P, Montes I, Miranda E (2016) Bivariate {$p$}-boxes. Internat. J. Uncertain. Fuzziness Knowledge-Based Systems 24:229--263.
  \url{https://doi.org/10.1142/S0218488516500124}.

\bibitem{SamDibKleMes17}
Saminger-Platz S, Dibala M, Klement EP, Mesiar R (2017) Ordinal sums
  of binary conjunctive operations based on the product. Publ. Math. Debrecen 91:63--80. \url{https://doi.org/10.5486/PMD.2017.7636}.

\bibitem{SamKolMesKle20}
Saminger-Platz S, Koles\'{a}rov\'{a} A, Mesiar R, Klement EP (2020) The key role of convexity in some copula constructions. Eur. J. Math. 6:533--560.
  \url{https://doi.org/10.1007/s40879-019-00346-3}.

\bibitem{Sca96}
Scarsini M (1996) Copulae of capacities on product spaces. Distributions with
  Fixed Marginals and Related Topics, {S}eattle, {WA}, 1993 (Inst. Math. Statist., Hayward, CA), 307--318.
  \url{https://doi.org/10.1214/lnms/1215452627}.

\bibitem{Schm23}
Schmelzer B (2023)
Random sets, copulas and related sets of probability measures.
Int. J. Approx. Reason. 160:108952. \url{https://doi.org/10.1016/j.ijar.2023.108952}.

\bibitem{Sto23}
Stopar N (2023)
Representation of the infimum and supremum of a family of
  multivariate distribution functions.
Fuzzy Sets and Systems 458:1--25. \url{https://doi.org/10.1016/j.fss.2022.05.001}.

\bibitem{SuaGil86}
Su\'arez~Garc\'{\i}a F, Gil~\'Alvarez P (1986) Two families of fuzzy
  integrals. Fuzzy Sets and Systems 18:67--81.
  \url{https://doi.org/10.1016/0165-0114(86)90028-X}.

\bibitem{Sug74}
Sugeno M (1974) Theory of fuzzy integrals and its applications. Ph.D. thesis, Tokyo Institute of Technology.

\bibitem{SugMur87}
Sugeno M, Murofushi T (1987) Pseudo-additive measures and integrals. J.
  Math. Anal. Appl. 122:197--222.
  \url{https://doi.org/10.1016/0022-247X(87)90354-4}.

\bibitem{Wal81}
Walley P (1981) Coherent lower (and upper) probabilities. Technical
  Report~22, Department of Statistics, University of Warwick.

\bibitem{Wal91}
Walley P (1991) Statistical Reasoning with Imprecise Probabilities. (Chapman \&
  Hall, London). \url{https://doi.org/10.1007/978-1-4899-3472-7}.

\bibitem{Wil07}
Williams PM (2007) Notes on conditional previsions. Int. J. Approx.
  Reason. 44:366--383. \url{https://doi.org/10.1016/j.ijar.2006.07.019}.

\end{thebibliography}
\end{document}